\documentclass[a4paper, british]{amsart}

\usepackage{libertine}
\usepackage[libertine]{newtxmath}

\usepackage{babel}
\usepackage{enumitem}
\usepackage{hyperref}
\usepackage[utf8]{inputenc}
\usepackage{newunicodechar}
\usepackage{mathtools}
\usepackage{varioref}
\usepackage[arrow,curve,matrix]{xy}

\usepackage{colortbl}
\usepackage{graphicx}
\usepackage{tikz}

\usepackage{mathrsfs}

\definecolor{linkred}{rgb}{0.7,0.2,0.2}
\definecolor{linkblue}{rgb}{0,0.2,0.6}

\numberwithin{figure}{section}

\usepackage[hyperpageref]{backref}

\setdescription{labelindent=\parindent, leftmargin=2\parindent}
\setitemize[1]{labelindent=\parindent, leftmargin=2\parindent}
\setenumerate[1]{labelindent=0cm, leftmargin=*, widest=iiii}

\newunicodechar{α}{\ensuremath{\alpha}}
\newunicodechar{β}{\ensuremath{\beta}}
\newunicodechar{χ}{\ensuremath{\chi}}
\newunicodechar{δ}{\ensuremath{\delta}}
\newunicodechar{ε}{\ensuremath{\varepsilon}}
\newunicodechar{Δ}{\ensuremath{\Delta}}
\newunicodechar{η}{\ensuremath{\eta}}
\newunicodechar{γ}{\ensuremath{\gamma}}
\newunicodechar{Γ}{\ensuremath{\Gamma}}
\newunicodechar{ι}{\ensuremath{\iota}}
\newunicodechar{κ}{\ensuremath{\kappa}}
\newunicodechar{λ}{\ensuremath{\lambda}}
\newunicodechar{Λ}{\ensuremath{\Lambda}}
\newunicodechar{ν}{\ensuremath{\nu}}
\newunicodechar{μ}{\ensuremath{\mu}}
\newunicodechar{ω}{\ensuremath{\omega}}
\newunicodechar{Ω}{\ensuremath{\Omega}}
\newunicodechar{π}{\ensuremath{\pi}}
\newunicodechar{Π}{\ensuremath{\Pi}}
\newunicodechar{φ}{\ensuremath{\phi}}
\newunicodechar{Φ}{\ensuremath{\Phi}}
\newunicodechar{ψ}{\ensuremath{\psi}}
\newunicodechar{Ψ}{\ensuremath{\Psi}}
\newunicodechar{ρ}{\ensuremath{\rho}}
\newunicodechar{σ}{\ensuremath{\sigma}}
\newunicodechar{Σ}{\ensuremath{\Sigma}}
\newunicodechar{τ}{\ensuremath{\tau}}
\newunicodechar{θ}{\ensuremath{\theta}}
\newunicodechar{Θ}{\ensuremath{\Theta}}
\newunicodechar{ξ}{\ensuremath{\xi}}
\newunicodechar{Ξ}{\ensuremath{\Xi}}
\newunicodechar{ζ}{\ensuremath{\zeta}}

\newunicodechar{ℓ}{\ensuremath{\ell}}
\newunicodechar{ï}{\"{\i}}

\newunicodechar{𝔸}{\ensuremath{\bA}}
\newunicodechar{𝔹}{\ensuremath{\bB}}
\newunicodechar{ℂ}{\ensuremath{\bC}}
\newunicodechar{𝔻}{\ensuremath{\bD}}
\newunicodechar{𝔼}{\ensuremath{\bE}}
\newunicodechar{𝔽}{\ensuremath{\bF}}
\newunicodechar{𝔾}{\ensuremath{\bG}}
\newunicodechar{ℕ}{\ensuremath{\bN}}
\newunicodechar{ℙ}{\ensuremath{\bP}}
\newunicodechar{ℚ}{\ensuremath{\bQ}}
\newunicodechar{ℝ}{\ensuremath{\bR}}
\newunicodechar{𝕏}{\ensuremath{\bX}}
\newunicodechar{ℤ}{\ensuremath{\bZ}}
\newunicodechar{𝒜}{\ensuremath{\sA}}
\newunicodechar{ℬ}{\ensuremath{\sB}}
\newunicodechar{𝒞}{\ensuremath{\sC}}
\newunicodechar{𝒟}{\ensuremath{\sD}}
\newunicodechar{ℰ}{\ensuremath{\sE}}
\newunicodechar{ℱ}{\ensuremath{\sF}}
\newunicodechar{𝒢}{\ensuremath{\sG}}
\newunicodechar{ℋ}{\ensuremath{\sH}}
\newunicodechar{𝒥}{\ensuremath{\sJ}}
\newunicodechar{ℒ}{\ensuremath{\sL}}
\newunicodechar{𝒪}{\ensuremath{\sO}}
\newunicodechar{𝒬}{\ensuremath{\sQ}}
\newunicodechar{𝒮}{\ensuremath{\sS}}
\newunicodechar{𝒯}{\ensuremath{\sT}}
\newunicodechar{𝒲}{\ensuremath{\sW}}

\newunicodechar{∂}{\ensuremath{\partial}}
\newunicodechar{∇}{\ensuremath{\nabla}}

\newunicodechar{↺}{\ensuremath{\circlearrowleft}}
\newunicodechar{∞}{\ensuremath{\infty}}
\newunicodechar{⊕}{\ensuremath{\oplus}}
\newunicodechar{⊗}{\ensuremath{\otimes}}
\newunicodechar{•}{\ensuremath{\bullet}}
\newunicodechar{Λ}{\ensuremath{\wedge}}
\newunicodechar{↪}{\ensuremath{\into}}
\newunicodechar{→}{\ensuremath{\to}}
\newunicodechar{↦}{\ensuremath{\mapsto}}
\newunicodechar{⨯}{\ensuremath{\times}}
\newunicodechar{∪}{\ensuremath{\cup}}
\newunicodechar{∩}{\ensuremath{\cap}}
\newunicodechar{⊋}{\ensuremath{\supsetneq}}
\newunicodechar{⊇}{\ensuremath{\supseteq}}
\newunicodechar{⊃}{\ensuremath{\supset}}
\newunicodechar{⊊}{\ensuremath{\subsetneq}}
\newunicodechar{⊆}{\ensuremath{\subseteq}}
\newunicodechar{⊂}{\ensuremath{\subset}}
\newunicodechar{⊄}{\ensuremath{\not \subset}}
\newunicodechar{≥}{\ensuremath{\geq}}
\newunicodechar{≠}{\ensuremath{\neq}}
\newunicodechar{≫}{\ensuremath{\gg}}
\newunicodechar{≪}{\ensuremath{\ll}}

\newunicodechar{≤}{\ensuremath{\leq}}
\newunicodechar{∈}{\ensuremath{\in}}
\newunicodechar{∉}{\ensuremath{\not \in}}
\newunicodechar{∖}{\ensuremath{\setminus}}
\newunicodechar{◦}{\ensuremath{\circ}}
\newunicodechar{°}{\ensuremath{^\circ}}
\newunicodechar{…}{\ifmmode\mathellipsis\else\textellipsis\fi}
\newunicodechar{·}{\ensuremath{\cdot}}
\newunicodechar{⋯}{\ensuremath{\cdots}}
\newunicodechar{∅}{\ensuremath{\emptyset}}
\newunicodechar{⇒}{\ensuremath{\Rightarrow}}

\newunicodechar{⁰}{\ensuremath{^0}}
\newunicodechar{¹}{\ensuremath{^1}}
\newunicodechar{²}{\ensuremath{^2}}
\newunicodechar{³}{\ensuremath{^3}}
\newunicodechar{⁴}{\ensuremath{^4}}
\newunicodechar{⁵}{\ensuremath{^5}}
\newunicodechar{⁶}{\ensuremath{^6}}
\newunicodechar{⁷}{\ensuremath{^7}}
\newunicodechar{⁸}{\ensuremath{^8}}
\newunicodechar{⁹}{\ensuremath{^9}}
\newunicodechar{ⁱ}{\ensuremath{^i}}

\newunicodechar{⌈}{\ensuremath{\lceil}}
\newunicodechar{⌉}{\ensuremath{\rceil}}
\newunicodechar{⌊}{\ensuremath{\lfloor}}
\newunicodechar{⌋}{\ensuremath{\rfloor}}

\newunicodechar{≅}{\ensuremath{\cong}}
\newunicodechar{⇔}{\ensuremath{\Leftrightarrow}}
\newunicodechar{∃}{\ensuremath{\exists}}
\newunicodechar{±}{\ensuremath{\pm}}

\DeclareFontFamily{OMS}{rsfs}{\skewchar\font'60}
\DeclareFontShape{OMS}{rsfs}{m}{n}{<-5>rsfs5 <5-7>rsfs7 <7->rsfs10 }{}
\DeclareSymbolFont{rsfs}{OMS}{rsfs}{m}{n}
\DeclareSymbolFontAlphabet{\scr}{rsfs}
\DeclareSymbolFontAlphabet{\scr}{rsfs}

\DeclareFontFamily{U}{mathx}{\hyphenchar\font45}
\DeclareFontShape{U}{mathx}{m}{n}{
      <5> <6> <7> <8> <9> <10>
      <10.95> <12> <14.4> <17.28> <20.74> <24.88>
      mathx10
      }{}
\DeclareSymbolFont{mathx}{U}{mathx}{m}{n}
\DeclareFontSubstitution{U}{mathx}{m}{n}
\DeclareMathAccent{\wcheck}{0}{mathx}{"71}

\DeclareMathOperator{\Aut}{Aut}
\DeclareMathOperator{\codim}{codim}

\DeclareMathOperator{\Id}{Id}

\DeclareMathOperator{\img}{img}
\DeclareMathOperator{\Pic}{Pic}
\DeclareMathOperator{\rank}{rank}

\DeclareMathOperator{\reg}{reg}

\DeclareMathOperator{\sEnd}{\sE\negthinspace \mathit{nd}}
\DeclareMathOperator{\sing}{sing}

\DeclareMathOperator{\Sym}{Sym}
\DeclareMathOperator{\supp}{supp}
\DeclareMathOperator{\tor}{tor}

\newcommand{\sA}{\scr{A}}
\newcommand{\sB}{\scr{B}}
\newcommand{\sC}{\scr{C}}
\newcommand{\sD}{\scr{D}}
\newcommand{\sE}{\scr{E}}
\newcommand{\sF}{\scr{F}}
\newcommand{\sG}{\scr{G}}
\newcommand{\sH}{\scr{H}}

\newcommand{\sJ}{\scr{J}}

\newcommand{\sL}{\scr{L}}

\newcommand{\sO}{\scr{O}}

\newcommand{\sQ}{\scr{Q}}

\newcommand{\sS}{\scr{S}}
\newcommand{\sT}{\scr{T}}
\newcommand{\sU}{\scr{U}}

\newcommand{\sW}{\scr{W}}

\newcommand{\bA}{\mathbb{A}}
\newcommand{\bB}{\mathbb{B}}
\newcommand{\bC}{\mathbb{C}}
\newcommand{\bD}{\mathbb{D}}
\newcommand{\bE}{\mathbb{E}}
\newcommand{\bF}{\mathbb{F}}
\newcommand{\bG}{\mathbb{G}}

\newcommand{\bN}{\mathbb{N}}

\newcommand{\bP}{\mathbb{P}}
\newcommand{\bQ}{\mathbb{Q}}
\newcommand{\bR}{\mathbb{R}}

\newcommand{\bX}{\mathbb{X}}

\newcommand{\bZ}{\mathbb{Z}}

\theoremstyle{plain}
\newtheorem{thm}{Theorem}[section]

\newtheorem{defn}[thm]{Definition}
\newtheorem{fact}[thm]{Fact}
\newtheorem{lem}[thm]{Lemma}

\newtheorem{prop}[thm]{Proposition}

\theoremstyle{remark}

\newtheorem{asswlog}[thm]{Assumption w.l.o.g.}
\newtheorem{claim}[thm]{Claim}
\newtheorem{c-n-d}[thm]{Claim and Definition}
\newtheorem{consequence}[thm]{Consequence}
\newtheorem{construction}[thm]{Construction}

\newtheorem{notation}[thm]{Notation}

\newtheorem{rem}[thm]{Remark}

\newtheorem*{rem-nonumber}{Remark}

\numberwithin{equation}{thm}

\setlist[enumerate]{label=(\thethm.\arabic*), before={\setcounter{enumi}{\value{equation}}}, after={\setcounter{equation}{\value{enumi}}}}

\newcommand{\into}{\hookrightarrow}

\newcommand{\wtilde}{\widetilde}
\newcommand{\what}{\widehat}

\newcommand{\factor}[2]{\left. \raise 2pt\hbox{$#1$} \right/\hskip -2pt\raise -2pt\hbox{$#2$}}
\author{Daniel Greb} %
\address{Daniel Greb, Essener Seminar für Algebraische Geometrie und Arithmetik, Fakultät für Mathe\-ma\-tik, Universität Duisburg--Essen, 45117 Essen, Germany}
\email{\href{mailto:daniel.greb@uni-due.de}{daniel.greb@uni-due.de}}
\urladdr{\href{https://www.esaga.uni-due.de/daniel.greb/}{https://www.esaga.uni-due.de/daniel.greb}}
\newcommand{\Publication}[1]{}

\newcommand{\subversionInfo}{}
\newcommand{\svnid}[1]{}
\newcommand{\approvals}[2][Approval]{}
\usepackage{tikz-cd}
\usepackage{nicematrix}

\DeclareMathAccent{\wcheck}{0}{mathx}{"71}

\author{Stefan Kebekus} %
\address{Stefan Kebekus, Mathematisches Institut, Albert-Ludwigs-Universität Freiburg, Ernst-Zermelo-Straße 1, 79104 Freiburg im Breisgau, Germany \&
  Freiburg Institute for Advanced Studies (FRIAS), Freiburg im Breisgau, Germany}
\email{\href{mailto:stefan.kebekus@math.uni-freiburg.de}{stefan.kebekus@math.uni-freiburg.de}}
\urladdr{\href{https://cplx.vm.uni-freiburg.de}{https://cplx.vm.uni-freiburg.de}}

\author{Thomas Peternell} %
\address{Thomas Peternell, Mathematisches Institut, Universität
  Bayreuth, 95440~Bayreuth, Germany}
\email{\href{mailto:thomas.peternell@uni-bayreuth.de}{thomas.peternell@uni-bayreuth.de}}
\urladdr{\href{https://www.komplexe-analysis.uni-bayreuth.de}{https://www.komplexe-analysis.uni-bayreuth.de}}


\makeatletter
\@namedef{subjclassname@2020}{2020 Mathematics Subject Classification}
\makeatother

\thanks{Daniel Greb is partially supported by the DFG-Research Training Group
  ``Symmetries and classifying spaces: analytic, arithmetic and derived''.
  Stefan Kebekus gratefully acknowledges partial support through a fellowship of
  the Freiburg Institute of Advanced Studies (FRIAS).  Thomas Peternell is
  partially supported by a DFG grant ``Zur Positivität in der Komplexen
  Geometrie'' and gratefully also acknowledges the support by FRIAS}

\title{Projectively flat KLT varieties}
\date{\today}

\makeatletter
\hypersetup{
  pdfauthor={\authors},
  pdftitle={\@title},
  pdfsubject={\@subjclass},
  pdfkeywords={\@keywords},
  pdfstartview={Fit},
  pdfpagelayout={TwoColumnRight},
  pdfpagemode={UseOutlines},
  bookmarks,
  colorlinks,
  linkcolor=linkblue,
  citecolor=linkred,
  urlcolor=linkred
}
\makeatother

\DeclareMathOperator{\alb}{alb}
\DeclareMathOperator{\Alb}{Alb}
\DeclareMathOperator{\Def}{Def}
\DeclareMathOperator{\diff}{d}
\DeclareMathOperator{\Div}{Div}
\DeclareMathOperator{\Gal}{Gal}
\DeclareMathOperator{\GL}{GL}
\DeclareMathOperator{\iit}{iit}
\DeclareMathOperator{\Iit}{Iit}
\DeclareMathOperator{\Lie}{Lie}
\DeclareMathOperator{\loc}{loc}
\DeclareMathOperator{\PGL}{ℙ\negthinspace\GL}
\DeclareMathOperator{\Rad}{Rad}
\DeclareMathOperator{\sh}{sha}
\DeclareMathOperator{\Sh}{Sha}
\DeclareMathOperator{\SL}{SL}

\theoremstyle{remark}

\begin{document}

\begin{abstract}
In the context of uniformisation problems, we study projective varieties with
klt singularities whose cotangent sheaf admits a projectively flat structure
over the smooth locus.  Generalising work of Jahnke-Radloff, we show that torus
quotients are the only klt varieties with semistable cotangent sheaf and
extremal Chern classes.  An analogous result for varieties with nef normalised
cotangent sheaves follows.
\end{abstract}
\approvals[Approval for abstract]{Daniel & yes\\Stefan & yes\\ Thomas & yes}

\maketitle
\tableofcontents

%
%
\svnid{$Id: 01-intro.tex 631 2020-09-23 09:23:01Z kebekus $}

\section{Introduction}
\subversionInfo

\subsection{Projective manifolds with projectively flat cotangent bundle}
\approvals{Daniel & yes \\Stefan & yes\\ Thomas & yes}

Let $ℰ$ be a locally free sheaf of rank $r$ on a complex, projective manifold
$X$ of dimension $n$.  If $ℰ$ is semistable with respect to some ample divisor
$H ∈ \Div(X)$, then the Bogomolov-Gieseker Inequality holds:
\begin{equation}\label{eq:x1}
  \frac{r-1}{2r} · c_1(ℰ)² · [H]^{n-2} ≤ c_2(ℰ) · [H]^{n-2}.
\end{equation}
In case of equality, $ℰ$ is projectively flat.  Motivated by the structure
theory of higher-dimensional projective manifolds, we are particularly
interested in the case where $ℰ$ is the cotangent bundle $Ω_X¹$ of an
$n$-dimensional manifold $X$.  In this setup, the equality case of \eqref{eq:x1}
reads
\begin{equation}\label{BGE}
  \frac{n-1}{2n} · c_1(X)² · [H]^{n-2} = c_2(X) · [H]^{n-2}.
\end{equation}

While semistability of $Ω¹_X$ occurs in many relevant cases and has important
geometric consequences, the Equality~\eqref{BGE} poses severe restrictions on
the geometry of $X$.
\begin{itemize}
\item In case $K_X$ is ample, equality never happens, owing to the stronger
  Miyaoka--Yau inequality.

\item If $K_X \equiv 0$, then by Yau's theorem, $X$ is an étale quotient of an
  Abelian variety.

\item If $X$ is Fano and Kähler-Einstein, again the equality \eqref{BGE} cannot
  occur, owing to the Chen--Oguie inequality.
\end{itemize}
In their remarkable paper \cite{MR3030068}, Jahnke and Radloff proved the
following complete characterisation of manifolds with semistable cotangent
bundle for which Equality~\eqref{BGE} holds.

\begin{thm}[\protect{Characterisation of torus quotients, \cite[Thms.~0.1 and 1.1]{MR3030068}}]
  Let $X$ be a projective manifold of dimension $n$ and assume that $Ω¹_X$ is
  $H$-semistable for some ample line bundle $H$.  If Equality~\eqref{BGE} holds,
  then $X$ is a finite étale quotient of a torus.  \qed
\end{thm}

In particular, this deals with manifolds of intermediate Kodaira dimension, and
moreover implies that in the Fano case Equality~\eqref{BGE} can never happen,
independent of the existence question for Kähler-Einstein metrics.

\subsection{Main result of this paper}
\approvals{Daniel & yes \\Stefan & yes\\ Thomas & yes}

We have learned from numerous previous results, including \cite{GKP13, LT18,
  GKT16, GKPT19, GKPT19b, GKP20a}, that the natural context for uniformisation
results is that of minimal model theory.  The aim of this paper is therefore to
generalise the theorem of Jahnke--Radloff to the case when $X$ has klt
singularities.

\begin{thm}[Characterisation of quasi-Abelian varieties]\label{thm:charQAbelian}
  Let $X$ be a projective klt space of dimension $n ≥ 2$ and let $H ∈ \Div(X)$
  be ample.  Assume that $Ω^{[1]}_X$ is semistable with respect to $H$ and that
  its $ℚ$-Chern classes satisfy the equation
  \begin{equation}\label{eq:qcce}
    \frac{n-1}{2n}·\what{c}_1 \left( Ω^{[1]}_X \right)² · [H]^{n-2} %
    = \what{c}_2 \left( Ω^{[1]}_X \right) · [H]^{n-2}.
  \end{equation}
  Then, $X$ is quasi-Abelian and has at worst quotient singularities.
\end{thm}

In Theorem~\ref{thm:charQAbelian}, we say that a normal projective variety $X$
is \emph{quasi-Abelian} if there exists a quasi-étale cover $\wtilde{X} → X$
from an Abelian variety $\wtilde{X}$ to $X$.  The symbols $\what{c}_{•}$ denote
$ℚ$-Chern classes on the klt variety $X$, as recalled in
\cite[Sect.~3]{GKPT19b}.

\subsection{Normalised cotangent sheaves}
\approvals{Daniel & yes \\Stefan & yes\\ Thomas & yes}

Work of Narasimhan, Seshadri and others, summarised for example in
\cite[Thm.~1.1]{MR3030068} and explained in detail by Nakayama in
\cite[Thm.~A]{NakayamaNormalizedPreprint}, can be used to reformulate
Jahnke-Radloff's result in terms of positivity properties of natural tensor
sheaves: a projective manifold $X$ of dimension $n$ is quasi-Abelian if and only
if the \emph{normalised cotangent bundle}, $\Sym^n Ω¹_X ⊗ 𝒪_X(-K_X)$, is nef.

While the arguments presented by Nakayama use intersection theory computations
on the total space of the projectivised cotangent bundle that cannot immediately
carried over to singular varieties, we are nevertheless able obtain the
analogous result in our setup.

\begin{thm}\label{thm:6-1}
  Let $X$ be a normal projective variety of dimension $n ≥ $.  Assume that $X$
  is klt and that the reflexive normalised cotangent sheaf
  \[
    \bigl(\Sym^nΩ¹_X ⊗ 𝒪_X(-K_X)\bigr)^{**}
  \]
  is nef.  Then, $X$ is quasi-Abelian.
\end{thm}

Definition~\vref{def:possheaf} recalls the meaning of \emph{nef} for a sheaf
that is not necessarily locally free.

\subsection{Strategy of proof}
\approvals{Daniel & yes \\Stefan & yes\\ Thomas & yes}

While the general strategy of proof is similar to the one employed by Jahnke and
Radloff, our argument introduces a number of new tools; this includes a detailed
analysis of sheaves that are projectively flat on the smooth a klt variety,
especially as their behaviour near the singularities is concerned.  On the one
hand, these tools allow to deal with the (serious) complications arising from
the singularities.  On the other hand, they enable us to streamline parts of
Jahnke-Radloff's proof, thereby clarifying the underlying geometric principles.

As an intermediate step, we obtain the following result, which might be of
independent interest to some of the readers.

\begin{thm}[\protect{= Theorem~\ref{split} in Section~\ref{sect:split}}]
  Let $X$ be a normal projective klt variety.  Assume that the sheaf $Ω^{[1]}_X$
  of reflexive differentials is of the form
  \begin{equation}
    Ω^{[1]}_X ≅ ℒ^{⊕ n}
  \end{equation}
  where $ℒ$ is a reflexive sheaf of rank one, so that in particular $ℒ$ is
  $ℚ$-Cartier.  If $ℒ$ is pseudo-effective, then $K_X \sim_ℚ 0$.
\end{thm}

\subsection{Thanks}
\approvals{Daniel & yes \\Stefan & yes\\ Thomas & yes}

We would like to thank Indranil Biswas and Stefan Schröer for answering our
questions.

%
%
\svnid{$Id: 02-notation.tex 625 2020-09-23 07:51:32Z peternell $}

\section{Conventions, notation, and variations of standard facts}
\subversionInfo

\subsection{Global conventions}
\approvals{Daniel & yes \\Stefan & yes\\ Thomas & yes}

Throughout this paper, all schemes, varieties and morphisms will be defined over
the complex number field.  We follow the notation and conventions of
Hartshorne's book \cite{Ha77}.  In particular, varieties are always assumed to
be irreducible.  For all notation around Mori theory, such as klt spaces and klt
pairs, we refer the reader to \cite{KM98}.

\subsection{Varieties and complex spaces}
\approvals{Daniel & yes \\Stefan & yes\\ Thomas & yes}

In order to keep notation simple, we will sometimes, when there is no danger of
confusion, not distinguish between algebraic varieties and their underlying
complex spaces.  Along these lines, if $X$ is a quasi-projective complex
variety, we write $π_1(X)$ for the fundamental group of the associated complex
space.

\subsection{Reflexive sheaves}
\approvals{Daniel & yes \\Stefan & yes\\ Thomas & yes}

As in most other papers on the subject, we will frequently consider reflexive
sheaves and take reflexive hulls.  Given a normal, quasi-projective variety (or
normal, irreducible complex space) $X$, we write
$Ω^{[p]}_X := \bigl(Ω^p_X \bigr)^{**}$ and refer to this sheaf as \emph{the
  sheaf of reflexive differentials}.  More generally, given any coherent sheaf
$ℰ$ on $X$, write $ℰ^{[⊗m]} := \bigl(ℰ^{⊗m} \bigr)^{**}$ and
$\det ℰ := \bigl( Λ^{\rank ℰ} ℰ \bigr)^{**}$.  Given any morphism $f : Y → X$ of
normal, quasi-projective varieties (or normal, irreducible, complex spaces), we
write $f^{[*]} ℰ := (f^* ℰ)^{**}$.

\subsection{Nef sheaves}
\approvals{Daniel & yes \\Stefan & yes\\ Thomas & yes}

We recall the notion of a nef sheaf in brief and collect basic properties.

\begin{defn}[Nef and ample sheaves, \cite{MR670921}]\label{def:possheaf}
  Let $X$ be a normal, projective variety and let $𝒮 \not = 0$ be a non-trivial
  coherent sheaf on $X$.  We call $𝒮$ \emph{ample}/\emph{nef} if the locally
  free sheaf $𝒪_{ℙ(𝒮)}(1) ∈ \Pic(ℙ(𝒮))$ is ample/nef.
\end{defn}

We refer the reader to \cite{GrCa60} for the definition of $ℙ(𝒮)$, and to
\cite[Sect.~2 and Thm.~2.9]{MR670921} for a more detailed discussion of
amplitude and for further references.  We mention a few elementary facts without
proof.

\begin{fact}[Nef sheaves]\label{fact:1}
  Let $X$ be a normal, projective variety.
  \begin{enumerate}
  \item\label{il:f1-1} A direct sum of sheaves on $X$ if nef iff every summand
    is nef.
  \item\label{il:f1-2} Pull-backs and quotients of nef sheaves are nef.
  \item\label{il:f1-3} A sheaf $ℰ$ is nef on $X$ if and only if for every smooth
    curve $C$ and every morphism $γ: C → X$, the pull-back $γ^* ℰ$ is nef.  \qed
  \end{enumerate}
\end{fact}

\subsection{Flat sheaves}
\approvals{Daniel & yes \\Stefan & yes\\ Thomas & yes}

One key notion in our argument is that of a flat sheaf.  We briefly recall the
definition.

\begin{defn}[\protect{Flat sheaf, \cite[Def.  1.15]{GKP13}}]\label{defn:flat}
  If $X$ is any normal, irreducible complex space and $ℱ$ is any locally free
  coherent sheaf on $X$, we call $ℱ$ flat if it is defined by a
  (finite-dimensional) complex representation of the fundamental group $π_1(X)$.
  A locally free coherent sheaf on a quasi-projective variety is called flat if
  the associated analytic sheaf on the underlying complex space is flat.
\end{defn}

\begin{rem}[Simple properties of flat sheaves]\label{rem:2-3}
  Tensor powers, duals, symmetric products and wedge products of locally free,
  flat sheaves are locally free and flat.  The pull-back of a locally free, flat
  sheaf under an arbitrary morphism is locally free and flat.  If $ℱ$ is a
  locally free and flat sheaf on a normal, irreducible complex space $X$, then
  there exists a description of $ℱ$ in terms of a trivialising covering and
  transition functions where all transitions functions are constant.
\end{rem}

\begin{lem}[Chern class of flat sheaf]\label{lem:c1offlat}
  Let $ℱ$ be a locally free and flat sheaf on a normal, irreducible complex
  space $X$.  Then, its first Chern class $c_1(ℱ) ∈ H²\bigl( X,\, ℂ \bigr)$
  vanishes.
\end{lem}
\begin{proof}
  Recalling that $c_1(ℱ) = c_1(\det ℱ)$, Remark~\ref{rem:2-3} allows to assume
  without loss of generality that $ℱ$ is invertible.  Denote the sheaves of
  locally constant functions of $X$ by $\underline{•}$.  In the following
  diagram of exponential sequences,
  \[
    \begin{tikzcd}[column sep=1.7cm]
      0\ar[r] & \underline{ℤ} \ar[r, "\imath"] \ar[d, equal, "\Id_{\underline{ℤ}}"] & \underline{ℂ} \ar[r, "\exp(2π·•)"] \ar[d, hook] & \underline{ℂ^*} \ar[r] \ar[d, hook, "f"] & 0 \\
      0 \ar[r]& \underline{ℤ} \ar[r] & 𝒪_X \ar[r, "\exp(2π·•)"'] & 𝒪_X^* \ar[r] & 0
    \end{tikzcd}
  \]
  the two rows are exact, cf.~\cite[V.4]{MR580152}, so we obtain maps of exact
  row sequences as follows,
  \[
    \begin{tikzcd}[column sep=1cm]
      ⋯ \ar[r] & H¹\bigl(X,\, \underline{ℂ^*} \bigr) \ar[r] \ar[d, "H¹(f)"] & H² \bigl(X, \underline{ℤ} \bigr) \ar[r, "H²(\imath)"] \ar[d, equal, "H²(\Id_{\underline{ℤ}})"] & H²\bigl(X, \underline{ℂ} \bigr) \ar[r] \ar[d]& ⋯ \\
      ⋯ \ar[r] & H¹\bigl(X,\, 𝒪_X^* \bigr) \ar[r, "δ"'] & H² \bigl(X, \underline{ℤ} \bigr)\ar[r] & H²\bigl(X, 𝒪_X \bigr)\ar[r] & ⋯
    \end{tikzcd}
  \]
  We have seen in Remark~\ref{rem:2-3} that $ℱ$ can be defined by locally
  constant transition functions.  Its isomorphism class
  $[ℱ] ∈ H¹\bigl(X,\, 𝒪_X^*\bigr)$ will therefore lie in the image of $H¹(f)$.
  Commutativity of the preceding diagram together with exactness of the first
  row then shows that
  \[
    c_1(ℱ) = \bigl(H²(\imath) ◦ H²(\Id_{\underline{ℤ}}) ◦ δ\bigr)\bigl([ℱ]\bigr)
    = 0,
  \]
  as claimed.
\end{proof}

\subsection{Covering maps and quasi-étale morphisms}
\approvals{Daniel & yes \\Stefan & yes\\ Thomas & yes}

A \emph{cover} or \emph{covering map} is a finite, surjective morphism
$γ : X → Y$ of normal, quasi-projective varieties (or normal, irreducible
complex spaces).  The covering map $γ$ is called \emph{Galois} if there exists a
finite group $G ⊂ \Aut(X)$ such that $γ$ is isomorphic to the quotient map.

A morphism $f : X → Y$ between normal varieties (or normal, irreducible complex
spaces) is called \emph{quasi-étale} if $f$ is of relative dimension zero and
étale in codimension one.  In other words, $f$ is quasi-étale if
$\dim X = \dim Y$ and if there exists a closed, subset $Z ⊆ X$ of codimension
$\codim_X Z ≥ 2$ such that $f|_{X ∖ Z} : X ∖ Z → Y$ is étale.

\subsection{Maximally quasi-étale spaces}
\approvals{Daniel & yes \\Stefan & yes\\ Thomas & yes}
\label{ssec:mqes}

Let $X$ be a normal, quasi-projective variety (or a normal, irreducible complex
space).  We say that $X$ is maximally quasi-étale if the natural push-forward
map of fundamental groups,
\[
  π_1(X_{\reg}) \xrightarrow{(\text{incl})_*} π_1(X)
\]
induces an isomorphism between the profinite completions,
$\what{π}_1(X_{\reg}) ≅ \what{π}_1(X)$.

\begin{rem}
  Recall from \cite[0.7.B on p.~33]{FL81} that the natural push-forward map
  $(\text{incl})_*$ is always surjective.  If $X$ is any quasi-projective klt
  variety, then $X$ admits a quasi-étale cover that is maximally quasi-étale and
  again klt, \cite[Thm.~1.14]{GKP13}.
\end{rem}

\subsection{Local fundamental groups of contraction morphisms}
\approvals{Daniel & yes \\Stefan & yes\\ Thomas & yes}

The behaviour of the fundamental group under extremal contractions and
resolutions of singularities has been studied by Takayama in a series of papers.
The following result is due to him.

\begin{prop}[Local fundamental groups of contraction morphisms]\label{prop:takayama}
  Let $X$ be a projective klt variety and $f : X → Y$ the contraction of a
  $K_X$-negative extremal ray.  If $y ∈ Y$ is any point, then there exists a
  neighbourhood $U = U(y) ⊆ Y$, open in the analytic topology, such that
  $f^{-1}(U)$ is connected and simply connected.
\end{prop}
\begin{proof}
  Takayama formulates a global version of the result in
  \cite[Thm.~1.2]{Takayama2003}.  The arguments of \cite[Proof of Thm.~1.2 on
  p.~834]{Takayama2003} apply in our setup and give a full proof of
  Proposition~\ref{prop:takayama}.
\end{proof}

\subsection{Vanishing theorems}
\approvals{Daniel & yes \\Stefan & yes\\ Thomas & yes}

The following relative vanishing result is well-known in the algebraic case,
cf.~\cite[Thm.~1-2-5]{KMM87}.  Using a relative Kawamata-Viehweg vanishing
theorem for projective morphisms between complex spaces,
\cite[Thm.~3.7]{MR946250}, we find that is also works in the analytic category.

\begin{thm}[Relative vanishing]\label{thm:relVan}
  Let $X$ be a normal complex space with klt singularities, let $φ : X → Y$ be a
  projective morphism of complex spaces, and let $D$ be a $ℚ$-Cartier Weil
  divisor on $X$.  Assume that $D-K_X$ is $φ$-ample.  Then $Rⁱ φ_* 𝒪_X(D) = 0$,
  for all $i >0$.
\end{thm}
\begin{proof}
  The proof of \cite[Thm.~1-2-5]{KMM87} applies verbatim when we replace the
  reference \cite[Thm.~1-2-3]{KMM87} to the algebraic version of the relative
  Kawamata-Viehweg vanishing theorem by \cite[Thm.~3.7]{MR946250}, which in
  particular works for projective morphisms between complex manifolds and
  complex spaces.
\end{proof}

\subsection{Abundance}
\approvals{Daniel & yes \\Stefan & yes\\ Thomas & yes}

We will later use the following special case of the abundance conjecture.  Even
though the proof uses only standard results from the literature, we found it
worth to include it here in full, for the reader's convenience and for later
referencing.

\begin{prop}\label{prop:abundance}
  Let $f : X \dasharrow Y$ be a dominant, almost holomorphic, rational map with
  connected fibres between projective varieties.  Assume that
  \begin{enumerate}
  \item\label{il:AS1} The variety $Y$ is smooth, positive-dimensional and of
    general type.
    
  \item\label{il:AS2} The variety $X$ is klt, and the canonical divisor $K_X$ is
    nef.

  \item\label{il:AS3} The variety $X$ is smooth around general fibres of $f$.
    
  \item\label{il:AS4} The general fibre is smooth and is a good minimal model.
  \end{enumerate}
  Then, $X$ is a good minimal model.  In other words, $K_X$ is semiample.
\end{prop}
\begin{proof}
  Blowing up $X$ in the indeterminacy locus, we obtain a diagram
  \[
    \begin{tikzcd}
      \wtilde{X} \ar[r, "π"] \ar[d, "\wtilde{f}"'] & X \ar[d, dashed, "f"] \\
      Y \ar[r, equal] & Y
    \end{tikzcd}
  \]
  where $π$ is isomorphic around the general fibre $X_y ≅ \wtilde{X}_y$ of $f$
  or $\wtilde{f}$, respectively.  Next, consider the relative Iitaka fibration
  for $π^* K_X$ on $\wtilde{X}/Y$.  As discussed in \cite[Sect.~2.4]{MR3286535},
  this is given by the (standard) Iitaka fibration for a line bundle of the form
  \[
    ℒ := π^* K_X + \wtilde{f}^*(\text{sufficiently ample}).
  \]
  Since sufficiently high powers of $ℒ$ have no base points along the general
  fibre $\wtilde{X}_y$, we may replace $\wtilde{X}$ by a further blow-up and
  obtain a diagram as follows,
  \[
    \begin{tikzcd}
      \wtilde{X} \ar[r, equal] \ar[d, "\text{rel.~Iitaka}"'] & \wtilde{X} \ar[r, "π"] \ar[d, "\wtilde{f}"'] & X \ar[d, dashed, "f"] \\
      Z \ar[r] & Y \ar[r, equal] & Y,
    \end{tikzcd}
  \]
  where $Z$ is smooth and where $ℒ$ is numerically trivial on the general fibre
  $\wtilde{X}_z$ of $\wtilde{X}$ over $Z$.  The numerical dimension $ν(X)$ and
  the Kodaira dimension $κ(X)$ are then estimated in terms of $\wtilde{X}_y$ and
  $\dim Y$ as follows,
  \begin{align*}
    κ(\wtilde{X}_y) + \dim Y & = \dim Z && \text{\cite[Sect.~2.4]{MR3286535}} \\
                             & = ν(π^*K_X|_{\wtilde{X}_z}) + \dim Z && \text{since } π^*K_X|_{\wtilde{X}_z} \equiv 0\\
                             & ≥ ν(π^* K_X) = ν(K_X) && \text{\cite[V.~Lem.~2.3.(2)]{Nakayama04}} \\
                             & ≥ κ(K_X) && \text{\cite[Prop.~2.2]{MR782236}} \\
                             & ≥ κ(\wtilde{X}) && \text{because $π$ is birational}\\
                             & ≥ κ(\wtilde{X}_y) + \dim Y && \text{by \cite[Satz~III]{MR641815}}.
  \end{align*}
  It follows that $κ(X) = ν(X)$, and hence by \cite[Cor.~6-1-13]{KMM87} that
  $K_X$ is semiample, as desired.
\end{proof}

\subsection{The negativity lemma}
\approvals{Daniel & yes\\Stefan & yes\\ Thomas & yes}

For later reference, we note the following two minor generalisations of the
standard negativity lemma.

\begin{lem}[Negativity Lemma I]\label{lem:negativity1}
  Let $f : X → Y$ be a surjective morphism of normal, projective varieties.  If
  $E ∈ ℚ\Div(X)$ is any non-zero, effective, $ℚ$-Cartier $ℚ$-divisor on $X$ with
  $\codim_Y f(\supp E) ≥ 2$, then $E$ is not nef.
\end{lem}
\begin{proof}
  Choose a very ample line bundle $ℒ$ on $X$ and a general tuple of sections
  $(H_1, …, H_{\dim X-\dim Y}) ∈ |ℒ|^{⨯ (\dim X-\dim Y)}$.  Set
  $S := H_1 ∩ ⋯ ∩ H_{\dim X-\dim Y}$.  This is a normal subvariety of $X$.  The
  restricted morphism $f|_S$ is surjective and generically finite, and $E ∩ S$
  is a non-zero, effective, $ℚ$-Cartier $ℚ$-divisor on $S$.  The Stein
  factorisation of $f|_S$ is then birational and contracts $E ∩ S$.  We can thus
  apply the classic negativity lemma, \cite[Prop.~1.6]{Gra13}, to conclude that
  $E ∩ S$ (and therefore $E$) are non-nef.
\end{proof}

\begin{lem}[Negativity Lemma II]\label{lem:negativity2}
  Let $f : X → Y$ be a surjective morphism of normal, projective varieties where
  $Y$ is a curve.  If $E ∈ ℚ\Div(X)$ is any non-zero, effective, $ℚ$-Cartier
  $ℚ$-divisor on $X$ that maps to a point in $Y$, then either $κ(E) = 1$ or $E$
  is not nef.
\end{lem}
\begin{proof}
  Write $y := f(E)$.  If $\supp(E)$ equals the set-theoretic fibre $f^{-1}(y)$,
  then $κ(E) = 1$.  We will therefore assume that $\supp(E)$ is a proper subset
  of $f^{-1}(y)$.  Choose a very ample line bundle $ℒ$ on $X$ and a general
  tuple of sections $(H_1, …, H_{\dim X-2}) ∈ |ℒ|^{⨯ (\dim X-\dim Y)}$.  Set
  $S := H_1 ∩ ⋯ ∩ H_{\dim X-2}$.  This is a normal surface in $X$.  The divisor
  $E ∩ S$ is a non-zero, effective, $ℚ$-Cartier $ℚ$-divisor on $S$, supported on
  a proper subset of $(f|_S)^{-1}(y)$.  Standard computations with intersection
  numbers will then show that $E ∩ S$ (and therefore $E$) are non-nef.
\end{proof}

\subsection{Abelian group schemes}
\approvals{Daniel & yes\\Stefan & yes\\ Thomas & yes}

The proof of our main result, the characterisation of quasi-Abelian varieties in
Theorem~\ref{thm:charQAbelian}, uses a minor generalisation of Kollár's
characterisation of étale quotients of Abelian group schemes,
\cite[Thm.~6.3]{Kollar93}.  We recall the relevant notions first.

\begin{defn}[\protect{Abelian group scheme, \cite[(6.2.1) on p.~197]{Kollar93}}]
  An \emph{Abelian group scheme over a proper base} is a smooth morphism
  $a: A → B$ between smooth and proper varieties such that every fibre of $a$ is
  an Abelian variety, and such that there exists a section $B → A$.
\end{defn}

\begin{defn}[\protect{Generically large fundamental group, \cite[Defn.~6.1]{Kollar93}}]
  Let $X$ be a normal, projective variety and let $Y ⊂ X$ be a closed
  subvariety.  We say that \emph{$X$ has generically large fundamental group on
    $Y$} if for all very general points $y ∈ Y$ and for every closed and
  positive-dimensional subvariety $y ∈ Z ⊂ Y$ with normalisation $\overline{Z}$,
  the image of the natural morphism $π_1\bigl(\overline{Z} \bigr) → π_1(X)$ is
  infinite.
\end{defn}

The generalisation of Kollár's result is then formulated as follows.

\begin{prop}[Characterisation of étale quotients of Abelian group schemes]\label{prop:K63}
  Let $f : X → Y$ be a surjective morphism with connected fibres between normal,
  projective varieties.  Assume that there exists $Δ ∈ ℚ\Div(X)$ such that
  $(X,Δ)$ is klt.  Let $y ∈ Y$ be a very general point with fibre $X_y$ and
  assume that the following holds.
  \begin{itemize}
  \item The fibre $X_y$ has a finite étale cover that is birational to an
    Abelian variety.
    
  \item The variety $X$ has generically large fundamental group on $X_y$.
  \end{itemize}
  Then, there exists an étale cover $γ_X : \what{X} → X$ such that the fibration
  $\what{a} : \what{A} → \what{Y}$ obtained as the Stein factorisation of
  $(f◦γ_X) : \what{X} → Y$ is birational to an Abelian group scheme over a
  proper base.
\end{prop}
\begin{proof}
  If $X$ and $Y$ are smooth, this is \cite[Thm.~6.3]{MR1341589}.  If not, choose
  $Δ ∈ ℚ\Div(X)$ such that $(X,Δ)$ is klt.  We will consider resolutions of the
  singularities and construct a diagram of surjective morphisms between normal,
  projective varieties,
  \begin{equation}\label{eq:D2a}
    \begin{tikzcd}[column sep=2.5cm]
      \wcheck{X} \ar[r, "γ_{\wtilde{X}}\text{, finite and étale}"] \ar[d, "\wcheck{a} \text{, conn.~fibres}"'] & \wtilde{X} \ar[r, "π_X \text{, resolution}"] \ar[d, "\wtilde{f} \text{, conn.~fibres}"] & X \ar[d, "f \text{, conn.~fibres}"] \\
      \wcheck{Y} \ar[r, "γ_{\wtilde{Y}}\text{, finite}"'] & \wtilde{Y} \ar[r, "π_Y \text{, resolution}"'] & Y,
    \end{tikzcd}
  \end{equation}
  where the morphism $\wcheck{a}$ is birational to an Abelian group scheme over
  a proper base.  For this, begin by choosing appropriate resolution morphisms
  $π_X$ and $π_Y$ to construct the right square in \eqref{eq:D2a}.  Once this is
  done, choose a very general point $\wtilde{y} ∈ \wtilde{Y}$ with very general
  image point $y := π_Y(\wtilde{y})$ and observe that
  $\bigl( X_y, Δ|_{X_y}\bigr)$ is klt to infer from
  \cite[Thm.~1.1]{Takayama2003} that the natural morphisms
  \[
    (π_X)_*: π_1\bigl(\wtilde{X}\,\bigr) → π_1(X) \quad\text{and}\quad \bigl(π_X|_{\wtilde{X}_{\wtilde{y}}}\bigr)_*: π_1\Bigl(\wtilde{X}_{\wtilde{y}}\Bigr) → π_1(X_y)
  \]
  are isomorphic.  Kollár's result therefore applies to $\wtilde{f}$ and gives
  the left square.

  There is more: Takayama's result also shows that the étale morphism
  $γ_{\wtilde{X}}$ comes from an étale cover of $X$, say $γ_X : \what{X} → X$.
  Stein factorising $π_X◦γ_{\wtilde{X}}$ and $π_Y◦γ_{\wtilde{Y}}$, we will
  therefore obtain a diagram of the following form,
  \begin{equation}\label{eq:D2b}
    \begin{tikzcd}[column sep=2.5cm]
      \wcheck{X} \ar[rr, "π_X◦γ_{\wtilde{X}}", bend left=15] \ar[r, "π_{\what{X}}\text{, resolution}"'] \ar[d, "\wcheck{a} \text{, conn.~fibres}"'] & \what{X} \ar[r, "γ_X \text{, étale}"'] \ar[d, "\what{a}", dashed] & X \ar[d, "f \text{, conn.~fibres}"] \\
      \wcheck{Y} \ar[rr, "π_Y◦γ_{\wtilde{Y}}"', bend right=15] \ar[r, "π_{\wtilde{Y}}\text{, conn.~fibres}"] & \what{Y} \ar[r, "γ_Y \text{, finite}"] & Y,
    \end{tikzcd}
  \end{equation}
  where $\what{a} := π_{\wtilde{Y}}◦\wcheck{a}◦π^{-1}_{\what{X}}$.  This map
  makes \eqref{eq:D2b} commute and is birational to $\wcheck{a}$, which is in
  turn birational to an Abelian group scheme.  To end the proof, it will
  therefore suffice to show that $\what{a}$ is actually a morphism.  This is not
  so hard: if $\what{x} ∈ \what{X}$ is any point with preimage
  $Z := π^{-1}_{\what{X}} \bigl( \what{x}\, \bigr)$, then
  \[
    \bigl( γ_Y◦ π_{\wtilde{Y}} ◦ \wcheck{a}\,\bigr)(Z) =
    f\bigl(γ_X(\what{x}\,)\bigr) = \text{point in } Y
  \]
  but since $γ_Y$ is finite, this implies that
  $\bigl( π_{\wtilde{Y}} ◦ \wcheck{a}\,\bigr)(Z)$ is a point in $\what{Y}$ and
  hence that $\what{a}$ is well-defined at $\what{x}$.
\end{proof}

%
%
\svnid{$Id: 03-projectiveFlatness.tex 632 2020-09-23 09:24:31Z kebekus $}

\section{Projective flatness}
\subversionInfo

\subsection{Projectively flat bundles and sheaves}
\approvals{Daniel & yes \\Stefan & yes\\ Thomas & yes}

\emph{Projective flatness} is the core notion of this paper.  We will only
recall the most relevant definition here and refer the reader to
\cite[Sect.~3]{GKP20a} for a detailed discussion of projectively flat bundles
and sheaves.

\begin{defn}[Projectively flat bundles and sheaves on complex spaces]\label{def:3-2}
  Let $X$ be a normal and irreducible complex space, let $r ∈ ℕ$ be any number
  and let $ℙ → X$ be a locally trivial $ℙ^r$-bundle.  We call the bundle $ℙ → X$
  \emph{(holomorphically) projectively flat} if there exists a representation of
  the fundamental group, $ρ : π_1(X) → ℙ\negthinspace\GL(r+1,ℂ)$, and an
  isomorphism of complex spaces over $X$
  \[
    ℙ ≅_X \factor{\wtilde{X} ⨯ ℙ^r}{π_1(X)},
  \]
  where $\wtilde{X}$ is the universal cover of $X$ and where the action
  $π_1(X) ↺ \wtilde{X} ⨯ ℙ^r$ is the diagonal action.  A locally free sheaf $ℱ$
  of $𝒪_X$-modules is called \emph{(holomorphically) projectively flat} if the
  associated bundle $ℙ(ℱ)$ is projectively flat.
\end{defn}

\begin{defn}[Projectively flat bundles and sheaves on complex varieties]
  Let $X$ be a connected, complex, quasi-projective variety and let $r ∈ ℕ$ be
  any number.  An étale locally trivial $ℙ^r$-bundle $ℙ → X$ is called
  \emph{projectively flat} if the associated analytic bundle
  $ℙ^{(an)} → X^{(an)}$ is projectively flat.  Ditto for coherent sheaves.
\end{defn}

\subsection{Projectively flat bundles induced by a solvable representation}
\approvals{Daniel & yes \\Stefan & yes\\ Thomas & yes}

One way to analyse a projectively flat sheaf $ℰ$ on a manifold or variety $X$ is
to look at the Albanese map $X → \Alb(X)$, and to try and understand the
restriction $ℰ|_F$ to a general fibre.  The key insight here is that the
Albanese map is a Shafarevich map for the commutator subgroup,
\cite[Prop.~4.3]{Kollar93}.  This allows to relate the geometry of the Albanese
map to the group theory of the representation that defines the projectively flat
structure on $ℰ$.  The following proposition, which is a variation of material
that appears in \cite{MR3030068}, is key to our analysis.

\begin{prop}[Projectively flat bundles induced by a solvable representation]\label{prop:pfbisr}
  Let $X$ be a normal, irreducible complex space that is maximally quasi-étale.
  Let $ℱ$ be a rank-{}$n$ reflexive sheaf on $X$ such that $ℱ|_{X_{\reg}}$ is
  locally free and projectively flat, so that $ℙ(ℱ|_{X_{\reg}})$ is isomorphic
  to a projective bundle $ℙ_ρ$, for a representation
  $ρ : π_1(X_{\reg}) → \PGL(n,ℂ)$.  If the image of $ρ$ is contained in a
  connected, solvable, algebraic subgroup $G ⊂ \PGL(n,ℂ)$, then there exists an
  isomorphism
  \begin{equation}\label{eq:OmegaFA0-1}
    ℱ ≅ ℱ' ⊗ 𝒜.
  \end{equation}
  where $𝒜$ is Weil divisorial\footnote{Weil divisorial sheaf = reflexive sheaf
    of rank one} and where $ℱ'$ is locally free and admits a filtration
  \begin{equation}\label{eq:3-3-2}
    0 = ℱ'_0 ⊊ ℱ'_1 ⊊ ⋯ ⊊ ℱ'_n = ℱ'
  \end{equation}
  such that the following holds.
  \begin{enumerate}
  \item\label{il:x1} All sheaves $ℱ_i$ are locally free and flat in the sense of
    Definition~\ref{defn:flat}, with $\rank (ℱ_i)$ equal to $i$.
    
  \item\label{il:x2} The sheaf $ℱ_1$ is trivial, $ℱ_1 ≅ 𝒪_X$.
  \end{enumerate}
\end{prop}

\begin{rem}\label{rem:pfbisr}
  In the setting of Proposition~\ref{prop:pfbisr}, note the following
  \begin{enumerate}
  \item There exists an isomorphism of Weil divisorial sheaves,
    $\det ℱ ≅ 𝒜^{[⊗ n]} ⊗ \det ℱ'$.  Recall from Remark~\ref{rem:2-3} that the
    factor $\det ℱ'$ is invertible and numerically trivial.
    
  \item There exists an injection $𝒜 ≅ 𝒜 ⊗ ℱ'_1 ↪ ℱ$.
  \end{enumerate}
\end{rem}

\begin{proof}[Proof of Proposition~\ref{prop:pfbisr}]
  To begin, we claim that the projective representation $ρ$ lifts to a linear
  representation, which we will denote by $σ$.  To formulate our result more
  precisely, write $B(n,ℂ) ⊂ \GL(n,ℂ)$ for the subgroup of upper-triangular
  matrices and $B_1(n,ℂ) ⊂ B(n,ℂ)$ for the subgroup of upper-triangular matrices
  whose top-left entry equals one.  In other words,
  \[
    B_1(n,ℂ) := \left\{ \left.  \:\:
      \begin{pNiceMatrix}[nullify-dots]
        1 & * & \Cdots & & * \\
        0 & * & \Ddots & & \Vdots \\
        \Vdots & \Ddots & \Ddots & & \\
               & & & & * \\
        0 & \Cdots & & 0 & *
      \end{pNiceMatrix}
      \:\:\right|\:\: \text{all } * ∈ ℂ
    \right\} = \Bigl\{ (b_{••}) ∈ B(n, ℂ) \mid b_{11}=1\Bigr\}.
  \]
  With this notation, we claim there exists a connected linear group
  $\what{G} ⊆ B_1(n,ℂ) ⊂ \GL(n,ℂ)$ such that the natural projection
  $π : \GL(n,ℂ) → \PGL(n,ℂ)$ induces an isomorphism between $\what{G}$ and $G$.
  To this end, let $\what{G}' ⊂ \SL(n, ℂ)$ be the maximal connected subgroup of
  $π^{-1}(G) ∩ \SL(n, ℂ)$.  The restricted morphism
  $π|_{\what{G}'} : \what{G}' → G$ is then finite and surjective, and induces an
  isomorphism of the Lie algebras $\Lie(\what{G}')$ and $\Lie(G)$.  In
  particular, $\what{G}'$ is solvable.  Therefore, choosing an appropriate basis
  of $ℂ^{n}$, Lie's Theorem allows to assume without loss of generality that
  $\what{G}'$ is contained in the subgroup $B(n,ℂ)$ of upper-triangular
  matrices.  Finally, set $\what{G} := π^{-1}(G) ∩ B_1(n, ℂ)$.  The obvious fact
  that any invertible upper triangular matrix in
  $A ∈ \what{G}' ⊆ B(n,ℂ)$ has a unique multiple $λ·A$ that lies in
  $B_1(n, ℂ)$ implies that the restriction $π|_{\what{G}} : \what{G} → G$ is
  then isomorphic, as desired.  The following commutative diagram summarises the
  situation,
  \[
    \begin{tikzcd}[column sep=large]
      & \what{G} \arrow[r, hook] \ar[d, "≅"] & B_1(n,ℂ) \arrow[r, hook] & \GL(n,ℂ) \arrow[d, "π\text{, nat.~projection}"] \\
      π_1\Bigl(X^{(an)}\Bigr) \arrow[r, "ρ"'] \arrow[ur, "σ", bend left=20] & G \arrow[rr, hook] & & \PGL(n,ℂ).
    \end{tikzcd}
  \]
  
  If $ℱ'$ denotes the locally free, flat sheaf on $X$ associated with $σ$, then
  the factorisation of $σ$ via the group $B_1(n, ℂ)$ implies that $ℱ'$ admits
  a filtration as in \eqref{eq:3-3-2} that satisfies conditions \ref{il:x1} and
  \ref{il:x2}.  By construction, we have have
  $ℙ(ℱ|_{X_{\reg}}) ≅ ℙ(ℱ'|_{X_{\reg}})$, and so there exists an invertible
  sheaf $𝒜_{\reg}$ on $X_{\reg}$ such that
  $ℱ|_{X_{\reg}} ≅ ℱ'|_{X_{\reg}} ⊗ 𝒜_{\reg}$.  A presentation of $ℱ$ as in
  Equation~\eqref{eq:OmegaFA0-1} therefore exists by taking
  $𝒜 := ι_* 𝒜_{\reg}$, where $ι : X_{\reg} → X$ is the inclusion.
\end{proof}

%
%
\svnid{$Id: 04-splitting.tex 628 2020-09-23 08:00:35Z peternell $}

\section{Varieties with splitting cotangent sheaves}
\subversionInfo
\approvals{Daniel & yes \\Stefan & yes\\ Thomas & yes}

One relevant case that keeps appearing in the discussion of varieties with
projectively flat sheaf of reflexive differentials is the one where $Ω^{[1]}_X$
decomposes as a direct sum of Weil divisorial sheaves.  We discuss settings
where there are local or global decompositions of this form.  The results of
this section will be used later in the proof of our main result, but might be of
independent interest.  In relation to \cite{MR3030068}, the local results are
new, prompted by the appearance of singularities in our context, and the results
on varieties with globally split cotangent sheaf allow us to replace some of the
more involved arguments of \emph{loc.~cit.} by a clearer structure result.

\subsection{Varieties where $Ω^{[1]}_X$ is locally split}
\label{sect:split}
\approvals{Daniel & yes \\Stefan & yes\\ Thomas & yes}

Given a variety where $Ω^{[1]}_X$ is projectively flat, for instance because the
flatness criterion of \cite[Thm.~1.6]{GKP20a} applies, the local description of
projectively flat sheaves, \cite[Prop.~3.10]{GKP20a}, can often be used to
reduce to the case where $Ω^{[1]}_X$ decomposes locally.  We will show that this
already forces $X$ to have quotient singularities.

\begin{prop}\label{prop:icq}
  Let $X$ be a normal, irreducible complex space with klt singularities.  Assume
  that the sheaf $Ω^{[1]}_X$ of reflexive differentials is of the form
  \begin{equation}\label{eq:icq}
    Ω^{[1]}_X ≅ ℒ^{⊕ \dim X},
  \end{equation}
  where $ℒ$ is reflexive of rank one.  Then, $X$ has only isolated, cyclic
  quotient singularities.
\end{prop}

We prove Proposition~\ref{prop:icq} in Section~\vref{ssec:popicq}.

\begin{rem}
  If $X$ is a quasi-projective variety with klt singularities where $Ω^{[1]}_X$
  is of the form $ℒ^{⊕ \dim X}$ for a Weil-divisorial sheaf $ℒ$, then the proof
  of Proposition~\ref{prop:icq} applies verbatim to show that $X$ is
  Zariski-locally a quotient of a smooth variety.  We do not use this fact
  however.
\end{rem}

\subsection{Varieties where $Ω^{[1]}_X$ is globally split}
\approvals{Daniel & yes \\Stefan & yes\\ Thomas & yes}

There are settings where $Ω^{[1]}_X$ decomposes globally, for instance because
it is projectively flat and $X$ is known to be simply connected.  This obviously
puts strong conditions on the global geometry of $X$.  Two of these are
described here.

\begin{thm}\label{split}
  Let $X$ be a normal projective klt variety.  Assume that the sheaf $Ω^{[1]}_X$
  of reflexive differentials is of the form
  \begin{equation}\label{eq:decomposition_of_cotangent}
    Ω^{[1]}_X ≅ ℒ^{⊕ n}
  \end{equation}
  where $ℒ$ is Weil divisorial.  Then, $ℒ^{[⊗ n]} ≅ ω_X$, and $ℒ$ is in
  particular $ℚ$-Cartier.  If $ℒ$ is pseudo-effective\footnote{As usual, we call
    $ℒ$ pseudo-effective if some locally free power $ℒ^{[⊗ m]}$ is
    pseudo-effective}, then $K_X \sim_ℚ 0$.
\end{thm}

Theorem~\ref{split} is shown in Section~\vref{ssec:potsplit} below.  The
following much simpler result asserts that in case where the canonical class is
\emph{not} pseudo-effective, the variety $X$ can at least never be Fano.  We
include the short proof because we found it instructive, even though the result
is not used later.

\begin{thm}\label{Fano-split}
  Let $X$ be a normal, projective klt variety of dimension $n>1$.  Assume that
  there exists a $ℚ$-Cartier, Weil divisorial sheaf $ℒ$ such that
  $Ω^{[1]}_X ≅ ℒ^{⊕ n}$.  Then $X$ is not $ℚ$-Fano.  In other words, the
  anti-canonical divisor class $-K_X$ is not $ℚ$-ample.
\end{thm}
\begin{proof}
  We argue by contradiction and assume that $X$ \emph{is} Fano.  Choose a
  divisor $L$ for the Weil divisorial sheaf $ℒ$.  Given the linear equivalence
  $K_X \sim n·L$, the Fano assumption guarantees that $L-K_X$ is $ℚ$-ample.
  Therefore $H¹\bigl(X,\,ℒ \bigr) = 0$ by Kawamata-Viehweg vanishing,
  \cite[1-2-5]{KMM87}.  In particular, we find that
  $H¹ \bigl(X,\, Ω^{[1]}_X\bigr) = 0$, which is absurd, as this cohomology group
  contains the class of an ample divisor.  In more detail: Let $D$ be an ample
  Cartier divisor on $X$, let $π: \what{X} → X$ be a desingularisation and
  consider the Chern class
  $c_1\bigl(𝒪_{\what{X}}(π^*D)\bigr) ∈ H¹\bigl(\what{X},\, Ω¹_{\what{X}}\bigr)$.
  Since $π_*(Ω¹_{\what{X}}) = Ω^{[1]}_X$ by \cite[Thm.~1.1]{GKK08}, the Leray
  spectral sequence yields
  \[
    H¹\bigl(\what{X},\, Ω¹_{\what{X}}\bigr) = \underbrace{\vphantom{\Bigl(}H¹\bigl(X,\, Ω^{[1]}_X\bigr)}_{\text{classes that are trivial on all $π$-fibres}} ⊕
    \underbrace{\ker \Bigl(H⁰\bigr(X,\, R¹π_*Ω¹_{\what{X}}\bigl) → H²\bigl(X,\,
    Ω^{[1]}_X\bigr) \Bigr)}_{\text{classes that are non trivial on some $π$-fibres}}.
  \]
  Hence $D$ defines an non-zero class in $H¹\bigl(X,\, Ω^{[1]}_X\bigr)$,
  contradiction.
\end{proof}

\subsection{Proof of Proposition~\ref*{prop:icq}}
\label{ssec:popicq}
\approvals{Daniel & yes \\Stefan & yes\\ Thomas & yes}

The proof of Proposition~\ref{prop:icq} uses the following lemma, which might be
of independent interest.

\begin{lem}\label{lem:glcs}
  Let $X$ be a normal, complex space with only cyclic quotient singularities.
  Let $S ⊆ X_{\sing}$ be any irreducible component of the singular locus.  Then,
  there exists a non-empty, open set $S° ⊆ S$ such that every $x ∈ S°$ admits an
  open neighbourhood $U=U(x) ⊆ X$ of the form $U = W ⨯ V$, where $W$ is smooth
  and $V$ has an isolated quotient singularity.
\end{lem}
\begin{proof}
  To begin, consider the special case where $V$ is a complex vector space,
  $ρ : ℤ_m → \GL(V)$ is a linear representation of a cyclic group and where
  $X ⊆ V / ℤ_m$ is an open neighbourhood of $[\vec 0]$.  Recall that $ρ$
  decomposes into direct a sum of one-dimensional representations.  Identifying
  $ℤ_m$ with roots of unity, we can thus find linear coordinates on $V$ such
  that multiplication with $ξ ∈ ℤ_m$ is given as
  \[
    ξ · (v_1, …, v_m) = \bigl(ξ^{n_1}· v_1, …, ξ^{n_m}· v_m\bigr).
  \]
  An elementary computation in these coordinates shows that $S°$ exists, and can
  in fact be chosen to be dense in $S$.

  Returning to the general case where $X$ is arbitrary, let $s ∈ S$ be any
  point.  By assumption, there exists an open neighbourhood $U = U(x)$ and a
  cyclic cover $γ : \wtilde U → U$, say with group $G$.  Shrinking $U$ and
  passing to a subgroup, if need be, we may assume without loss of generality
  that $γ$ is totally branched over $s$, that is $γ^{-1}(s) = \{ \wtilde s \}$.
  The cyclic group thus acts on $\wtilde U$ and fixes $\wtilde s$.  The group
  $G$ clearly acts on the vector space $T_{\wtilde U}|_{\wtilde s}$ and
  linearisation at the fixed point, \cite[Proof of Théor\`eme 4]{MR0084174} or
  \cite[Cor.~2 on p.~13]{MR782881}, shows the existence of a $G$-invariant open
  neighbourhood $\wtilde U'$ of $\vec 0 ∈ T_{\wtilde U}|_{\wtilde s}$ and an
  equivariant, biholomorphic map between $\wtilde U$ and $\wtilde U$.  In other
  words, we are in the situation discussed in the first paragraph of this proof,
  where the claim has already been shown.
\end{proof}

\begin{proof}[Proof of Proposition~\ref{prop:icq}]
  As a first step in the proof of Proposition~\ref{prop:icq}, we claim that $ℒ$
  is $ℚ$-Cartier.  In fact, taking determinants of both sides in \eqref{eq:icq},
  we obtain that $ℒ^{[⊗ \dim X]} ≅ ω_X$, which is $ℚ$-Cartier by assumption.
  Choose a minimal number $N ∈ ℕ^+$ such that $ℒ^{[⊗ N]}$ is locally free.  The
  reflexive tensor product $\bigl( Ω^{[1]}_U \bigr)^{[⊗ N]}$ is then likewise
  locally free.

  As a second step, we show that $X$ has cyclic quotient singularities.  In
  fact, given any $x ∈ X$, we find an open neighbourhood $U=U(x)$ over which
  $ℒ^{[⊗ N]}$ is trivial.  Chose a trivialisation and let $γ: \wtilde{U} → U$ be
  the associated index-one cover, which is cyclic of order $N$.  The complex
  space $\wtilde{U}$ has again klt singularities and
  $γ^{[*]}\, ℒ ≅ 𝒪_{\wtilde{U}}$.  In particular, it follows that
  $Ω^{[1]}_{\wtilde{U}} ≅ γ^{[*]} Ω^{[1]}_U$ and $𝒯_{\wtilde U}$ are both free.
  The solution of the Lipman-Zariski conjecture for spaces with klt
  singularities\footnote{Notice that the proof of \cite[Thm.~3.8]{Druel13a}
    works in the analytic category without any changes.}, then asserts that
  $\wtilde{U}$ is smooth.  We conclude that $X$ has cyclic quotient
  singularities only.

  As a third and last step, we show that the singularities of $X$ are isolated.
  Assume to the contrary and let $S ⊆ X_{\sing}$ be a positive-dimensional,
  irreducible component.  We have seen in Lemma~\ref{lem:glcs} that there a
  non-empty, open subset $S° ⊆ S$ such that every $x ∈ S°$ admits an open
  neighbourhood $U = U(x) ⊆ X$ of product form $U = W ⨯ V$, where $W$ is smooth
  and $V$ has an isolated quotient singularity.  In particular,
  $\dim W = \dim S$ and
  \[
    Ω^{[1]}_U ≅ Ω¹_W ⊕ Ω^{[1]}_V.
  \]
  In particular, the reflexive tensor power $\bigl( Ω^{[1]}_U \bigr)^{[⊗ N]}$ is
  written as a direct sum of reflexive sheaves with
  $\bigl( Ω^{[1]}_W \bigr)^{[⊗ (N-1)]} ⊗ Ω^{[1]}_V$ as one of its direct
  summands.  It follows that $\bigl( Ω^{[1]}_U \bigr)^{[⊗ N]}$ is not locally
  free and accordingly, neither is $\bigl( Ω^{[1]}_X \bigr)^{[⊗ N]}$.  This
  contradicts the results obtained in the first paragraph of this proof, and
  therefore finishes the proof of Proposition~\ref{prop:icq}.
\end{proof}

\subsection*{Proof of Theorem~\ref*{split}}
\label{ssec:potsplit}
\approvals{Daniel & yes \\Stefan & yes\\ Thomas & yes}

We work in the setting of Theorem~\ref{split}.  This isomorphism
$ℒ^{[⊗ n]} ≅ ω_X$ (and thus the assertion that $ℒ$ is $ℚ$-Cartier) follows by
taking determinants on either side of \eqref{eq:decomposition_of_cotangent}.  It
will therefore make sense to assume that $ℒ$ is pseudo-effective, and we
maintain this assumption throughout.

Since $X$ has (isolated) finite quotient singularities by
Proposition~\ref{prop:icq}, we may apply \cite[Prop.~3.10]{GKPT19b} and choose a
\emph{Mumford cover} $γ: Y → X$, that is, a finite morphism from a normal
variety $Y$ to $X$ factoring through the local quotient presentations of the
singularities of $X$; see \cite[Sect.~3]{GKPT19b} for a detailed discussion of
such covers.  In particular, the reflexive pullback $γ^{[*]}ℱ$ of any Weil
divisorial sheaf $ℱ$ on $X$ is locally free on $Y$.  Let $π: \what{Y} → Y$ be a
log-resolution of singularities and set $τ = γ◦π$, so that we are considering a
commutative diagram
\[
  \begin{tikzcd}[column sep=3cm]
    \what{Y} \ar[d, "π\text{, log resolution}"'] \ar[rd, bend left=10, "τ = γ◦π"] \\
    Y \ar[r, "γ\text{, Mumford cover}"'] & X.
  \end{tikzcd}
\]
Finally, let $R$ be the subset of $\what{Y}$ where $τ$ is not étale.  We
decompose $R$ it into a divisorial part $R_1$ and non-divisorial part $R_2$, so
that $R = R_1 ∪ R_2$.

Since $γ$ is finite, $ℒ_Y := γ^{[*]}ℒ$ is pseudo-effective.  As $ℒ_Y$ is also
locally free by the observation made above, the pull back
$ℒ_{\what{Y}} := π^*ℒ_Y$ is a locally free, pseudo-effective sheaf on
$\what{Y}$.  We may hence choose a singular metric $h$ on the line bundle
associated to $ℒ_{\what{Y}}$ that is locally given by
$h(•) = e^{-\varphi(•)}|•|$, where $\varphi$ is plurisubharmonic and in
$L¹_{\loc}$.
  
The decomposition $Ω^{[1]}_X ≅ ℒ^{⊕ n}$ induces $n$ injective morphisms
$ι_j : ℒ → Ω_X^{[1]}$.  Using functorial pullback for differential forms on klt
spaces, \cite[Thm.~4.3]{GKKP11}, we obtain induced maps
\[
  \begin{tikzcd}
    τ^{[*]}ℒ \ar[r] & τ^{[*]}Ω_X^{[1]} \ar[rrr, "dτ\text{, pull-back}"] &&& Ω¹_{\what{Y}}.
  \end{tikzcd}
\]
Composing this with the natural map
\[
  \begin{tikzcd}
    ℒ_{\what{Y}} = π^*\bigl(γ^{[*]}ℒ\bigr) \ar[r] & (γ◦π)^{[*]} ℒ = τ^{[*]} ℒ
  \end{tikzcd}
\]
we obtain naturally induced morphisms
$\what{ι}_j : ℒ_{\what{Y}} → Ω¹_{\what{Y}}$, which are generically injective and
hence injective, as $ℒ_{\what{Y}}$ is locally free.  We view these morphisms as
sections
\[
  θ_j ∈ H⁰\bigl(\what{Y},\, Ω¹_{\what{Y}} ⊗ ℒ_{\what{Y}}^{-1} \bigr)
\]
and consider the subsheaves $𝒮_j ⊂ 𝒯_{\what{Y}}$ consisting of germs of
holomorphic vector fields $\vec{v}_j$ such that the contraction with $θ_j$
vanishes: $i_{\vec{v}_j}(θ_j) = 0$.

By a fundamental theorem of Demailly, \cite[p.~93, Main Thm.]{Dem02}, the
sheaves $𝒮_j$ are holomorphic foliations of codimension $1$ and the curvature of
$h$ vanishes along the leaves of $𝒮_j$ for all $j$.  More precisely, let
$Z ⊂ \what{Y}$ be the set of all those points at which at least one of the
foliations $𝒮_j$ becomes singular; this is an analytic subset of codimension at
least $2$.  Let $y ∈ \what{Y} ∖ Z$ and let $\varphi$ be a local plurisubharmonic
function defining the metric $h$.  If $B ⊂ X$ is the leaf through $y$ of one of
the foliations $𝒮_{•}$, then $B$ is smooth at $y$, locally given by the
vanishing of a local coordinate $z_1$ on $Y$.  Moreover, the computation done in
\cite[p.~97, last paragraph of the proof]{Dem02} says that
\begin{equation}\label{eq:vanishing_along_leaves}
  \frac{∂² \varphi} {∂ z_j ∂ \overline {z}_k}(y) = 0 %
  \quad \quad \text{ for all } j,k ∈ \{2, \dots, n\}.
\end{equation}
If in addition, $y ∈ Y∖(Z∪R)$, then the decomposition
\eqref{eq:decomposition_of_cotangent} together with the definition of $R$
implies that the tangent spaces to these leaves span the tangent space
$T_y\what{Y}$.  From the vanishing of the curvature along the leaves obtained in
\eqref{eq:vanishing_along_leaves}, we can therefore conclude that
\[
  \frac{∂² \varphi} {∂ z_j ∂ \overline {z}_k}(y) = 0 %
  \quad \quad \text{ for all } j,k ∈ \{1, \dots, n\},
\]
and hence the curvature current $Θ_h$ of $h$ vanishes on $\what{Y} ∖ (Z ∪ R)$.
Restricting to $\what{Y} ∖ R_1$ and applying \cite[Chap.~III,
Cor.~2.11]{DemaillyBook2012} to the normal $(1,1)$-current $Θ_h$ and the at
least two-codimensional subset $Z ∪ R_2$, we see that
$Θ_h|_{\what{Y} ∖ R_1} \equiv 0$.  In other words, $\supp (Θ_h) ⊂ R_1$.

Consequently, if we write the decomposition of $R_1$ into irreducible components
as $R_1 = \bigcup_{j=1}^k A_j$, then \cite[Chap.~III,
Cor.~2.14]{DemaillyBook2012} implies that the positive current $Θ_h$ admits a
decomposition into currents of integration of the form
\begin{equation}\label{eq:curvature_decomposition}
  Θ_h = \sum_{j=1}^k λ_j·[A_j], \quad \text{ where } λ_j ∈ ℝ^{≥ 0}.
\end{equation}
For dimension reasons, in each $A_j$ there exist points $p_j$ that are not
contained in $Z$.  The leaf $B_j$ of any of the foliations through $p_j$
intersects $A_j$ in a codimension-one subset $A_j ∩ B_j ⊂ B_j$.  Hence,
\eqref{eq:curvature_decomposition} contradicts the vanishing statement
\eqref{eq:vanishing_along_leaves}, unless all the $λ_j$ are equal to zero.  We
conclude that $Θ_h$ vanishes identically on $\what{Y}$.  Consequently, we find
that the equations
\[
  c_1(ℒ) · [H_1] ⋯ [H_{n-1}] = 0 %
  \quad\text{and}\quad %
  c_1(ℒ_{\what Y}) · [τ^*H_1] ⋯ [τ^*H_{n-1}] = 0
\]
hold for any choice of ample divisors $H_1, …, H_{n-1}$ on $X$.  We conclude
that $ℒ \equiv 0$ and therefore $K_X \equiv 0$.  The claimed $ℚ$-linear
triviality then follows from the special case of the Abundance conjecture proven
in \cite[Cor.~4.9]{Nakayama04}.  The proof of Theorem~\ref{split} is thus
completed.  \qed

%
%
\svnid{$Id: 05-JRproof.tex 656 2020-10-13 14:06:05Z peternell $}

\section{Proof of Theorem~\ref*{thm:charQAbelian}}
\subversionInfo

\subsection{Preparation for the proof}
\approvals{Daniel & yes \\Stefan & yes\\ Thomas & yes}

Using the criterion for projective flatness spelled out in
\cite[Thm.~1.6]{GKP20a}, the assumptions made in Theorem~\ref{thm:charQAbelian}
imply that $Ω¹_{X_{\reg}}$ is projectively flat, at least after going to a
quasi-étale cover.  This allows to apply following lemma in various settings
that appear throughout the proof of Theorem~\ref{thm:charQAbelian}.

\begin{lem}[Consequences of projectively flat cotangent bundle]\label{lem:6-5}
  Let $X$ be a smooth, quasi-projective variety of dimension $n$ where $Ω¹_X$ is
  projectively flat.
  \begin{enumerate}
  \item\label{il:6-5-1} If $K_X$ is nef, then $X$ does not contain rational
    curves.
    
  \item\label{il:6-5-2} Assume that there exists an immersion $η : F → X$ where
    $F$ is projective and smooth.  If $η^* K_X$ is nef and if $K_F \equiv 0$,
    then $F$ is quasi-Abelian.
  \end{enumerate}
\end{lem}
\begin{proof}
  We prove the statements separately.  Both proofs use the fact that every
  locally free, projectively flat sheaf on a simply connected space is
  isomorphic to a direct sum of the form $ℒ^{⊕\,•}$ where $ℒ$ is invertible,
  cf.~\cite[Prop.~3.11]{GKP20a}.
  
  As to the first statement, let $η: ℙ_1 → X$ be a non-constant map.  Since
  $η^*Ω¹_X$ is projectively flat, there is an integer $m$ such that
  $η^*Ω¹_X ≅ 𝒪_{ℙ_1}(m)^{⊕ n}$.  The canonical morphism $η^*Ω¹_X → Ω¹_{ℙ_1}$
  then yields $m ≤ -2$, contradicting the nefness of $η^*K_X$.

  As to the second statement, applying the classic Decomposition Theorem
  \cite[Thm.~1]{Bea83} and possibly passing to a finite étale cover, we may
  assume that $F$ is of product form, $F ≅ A ⨯ Z$, where $A$ is Abelian
  (possibly a point) and where $Z$ is simply connected with $K_Z = 0$.  We aim
  to prove that $Z$ is a point, argue by contradiction and assume that
  $m := \dim Z$ is positive.  Let $ι: Z → F$ denote the inclusion morphism.  The
  pull-back $(η◦ι)^* Ω¹_X$ is projectively flat and hence of the form $ℒ^{⊕ n}$,
  for a suitable line bundle $ℒ ∈ \Pic(Z)$.  The obvious surjection
  \[
    \begin{tikzcd}
      ℒ^{⊕ n} ≅ (η◦ι)^* Ω¹_X \ar[rr, two heads, "\diff (η◦ι)"] && Ω¹_Z
    \end{tikzcd}
  \]
  induces a non-trivial map $ℒ^{⊗m} → ω_Z ≅ 𝒪_Z$.  But then the assumption that
  $η^*K_X$ and hence also $ℒ$ are nef shows that $ℒ^{⊗m}$ cannot be a proper
  ideal sheaf, so $ℒ^{⊗m} ≅ 𝒪_Z$.  It follows that either $ℒ$ is torsion or that
  $H⁰\bigl(Z,\, Ω¹_Z\bigr) \ne 0$, both contradicting the simple connectedness
  of $Z$.
\end{proof}

\subsection{Proof of Theorem~\ref*{thm:charQAbelian}}
\label{ssec:pf44}
\approvals{Daniel & yes \\Stefan & yes\\ Thomas & yes}

We maintain notation and assumptions of Theorem~\ref{thm:charQAbelian}
throughout the present Section~\ref{ssec:pf44}.  The proof follows the strategy
of \cite{MR3030068} closely, but has to overcome a fair number of technical
difficulties arising from the presence of singularities.  It is fairly long and
therefore subdivided into numerous steps.  The main idea is to show abundance
for $X$, and then to analyse the Iitaka fibration.  Abundance is shown in Steps
4, 5 and 6, with an inductive argument using repeated covers, fibrations, and
restrictions to fibres.

\begin{itemize}
\item Steps 1--3 set the stage, prove minimality of $X$ and put limits on its
  possible numerical dimension and Kodaira dimension.

\item Step 4 fibres $X$ using a Shafarevich map construction.  The general fibre
  is called $F$.

\item Step 5, which is the longest step of this proof, considers a suitable
  cover of $\what{F}$ of $F$ and takes a general fibre of its Albanese map.  An
  analysis of this fibre will then give abundance for $F$.
  
\item Step 6 uses abundance for $F$ to prove abundance for $X$.
  
\item Steps 7--8 end the proof by showing that the Iitaka fibration of a
  suitable cover of $X$ is birational (and then isomorphic) to an Abelian group
  scheme over a proper base.
\end{itemize}

\subsection*{Step 1: Simplification}
\approvals{Daniel & yes \\Stefan & yes\\ Thomas & yes}

Since assumption and conclusion of Theorem~\ref{thm:charQAbelian} are stable
under quasi-étale covers, we may apply \cite[Thm.~1.5]{GKP13}, pass to a maximal
quasi-étale cover and assume that the following holds in addition.

\begin{asswlog}[$X$ is maximally quasi-étale]\label{asswlog:mqe}
  The variety $X$ is maximally quasi-étale.  In other words, the algebraic
  fundamental groups $\what{π}_1(X_{\reg})$ and $\what{π}_1(X)$ agree.
\end{asswlog}

With this assumption in place, the criterion for projective flatness,
\cite[Thm.~1.6]{GKP20a}, implies that $Ω¹_{X_{\reg}}$ is projectively flat.  As
we have seen, this has a number of interesting consequences.

\begin{consequence}[Extension of projectively flat bundles]\label{cons:5-3}
  The extension result for projectively flat bundles, \cite[Prop.~3.9]{GKP20a},
  allows to find a projectively flat $ℙ^{n-1}$-bundle $ℙ → X$ and an isomorphism
  of $X$-schemes, $ℙ|_{X_{\reg}} ≅_X ℙ\big( Ω¹_{X_{\reg}} \bigr)$.
\end{consequence}

\begin{consequence}[Quotient singularities]\label{consequence:isolated}
  The local description of projectively flat sheaves, \cite[Prop.~3.10]{GKP20a},
  and Proposition~\ref{prop:icq} imply that $X$ has only isolated cyclic
  quotient singularities.
\end{consequence}

\begin{notation}\label{not:5-5}
  We chose a projectively flat bundle $ℙ → X$ as in Consequence~\ref{cons:5-3}
  and maintain this choice throughout.  We write $τ: π_1(X) → \PGL(n,ℂ)$ for the
  representation that defines the projectively flat structure on $ℙ$.
\end{notation}

If $K_X$ is numerically trivial, then the Chern class equality \eqref{eq:qcce}
reduces to $\what{c}_2(X) · [H]^{n-2} = 0$, and we may apply
\cite[Thm.~1.2]{LT18} to conclude that $\what{X}$ is quasi-Abelian.  This allows
to make the following assumption.

\begin{asswlog}[$K_X \not \equiv 0$]\label{asswlog:knnt}
  The canonical class $K_X$ is not numerically trivial.
\end{asswlog}

\subsection*{Step 2: Minimality of $X$}
\approvals{Daniel & yes \\Stefan & yes\\ Thomas & yes}

We will show in Claim~\ref{claim:minimal} below that the variety $X$ is minimal.
The following claim, which discusses the geometry of a hypothetical extremal
contraction morphism, is a first step in this direction.

\begin{claim}[Extremal rays]\label{claim:min0}
  Let $φ: X → Y$ be a contraction of a $K_X$-negative extremal ray.  Then, the
  higher direct images of $Ω^{[1]}_X$ vanish.  In other words, we have
  $Rⁱφ_* Ω^{[1]}_X = 0$, for all $i>0$.
\end{claim}
\begin{proof}[Proof of Claim~\ref{claim:min0}]
  If $y ∈ Y$ is any point, Takayama's result on the local fundamental groups of
  $φ$, Proposition~\ref{prop:takayama}, allows to find a neighbourhood
  $U = U(y) ⊆ Y^{(an)}$, open in the analytic topology, such that
  $V := φ^{-1}(U) ⊆ X^{(an)}$ is connected and simply connected.  Using simple
  connectedness, the local description of projectively flat sheaves,
  \cite[Prop.~3.10]{GKP20a}, provides us with a reflexive, rank-one coherent
  analytic sheaf $ℒ$ on $V$ and an isomorphism $Ω^{[1]}_V ≅ ℒ^{⊕ n}$.  To prove
  Claim~\ref{claim:min0}, it will then suffice to show that
  \begin{equation}\label{eq:rvanL}
    Rⁱ(φ|_V)_* ℒ = 0, \quad \text{for all }i>0.
  \end{equation}
  If $V$ is quasi-projective, this is of course the classic vanishing theorem
  \cite[Thm.~1-2-5]{KMM87}; recall that $ℒ$ is $ℚ$-Cartier, as $ℒ^{[n]} ≅ ω_V$.
  In our setting where $V$ is merely open in the analytic topology,
  \eqref{eq:rvanL} will follow from Theorem~\ref{thm:relVan}, as soon as we show
  that the (analytic!) sheaf $ℒ$ is Weil-divisorial: we need to find an
  (integral, analytic) Weil divisor $D ∈ \Div(V)$ and an isomorphism
  $ℒ ≅ 𝒪_V(D)$.  But this is easy: we know that $Ω^{[1]}_X$ is globally
  algebraic on the projective variety $X$, which allows to find an ample Cartier
  divisor $D' ∈ \Div(X)$ and a global, non-trivial morphism
  $σ' : 𝒪_X(-D') → Ω^{[1]}_X$.  Restricting to $V$ and composing with a suitable
  projection $Ω^{[1]}_V → ℒ$, we find a non-trivial sheaf morphism
  $σ : 𝒪_V(-D'|_V) → ℒ$.  One possible choice of a divisor $D$ is then given as
  $D := (\text{Zeros of }σ)-D'|_V$.  \qedhere~\mbox{(Claim~\ref{claim:min0})}
\end{proof}

\begin{claim}[Minimality of $X$]\label{claim:minimal}
  The canonical divisor $K_X$ is nef.
\end{claim}
\begin{proof}[Proof of Claim~\ref{claim:minimal}]
  Argue by contradiction and assume that there exists a contraction of a
  $K_X$-negative extremal ray.  Choosing a resolution of $X$, we find a sequence
  of morphisms
  \[
    \begin{tikzcd}[column sep=3cm]
      \wtilde{X} \arrow[r, "π\text{, resolution}"] & X \arrow[r, "φ\text{, contraction}"] & Y.
    \end{tikzcd}
  \]
  Looking at the Grothendieck spectral sequence for the composed operator
  $(φ◦π)_* = φ_*π_*$, the five-term exact sequence of low degrees reads
  \[
    0 → R¹ φ_* \bigl(π_* Ω¹_{\wtilde{X}}\bigr) %
    → R¹ (φ ◦ π)_*\: Ω¹_{\wtilde{X}} %
    \xrightarrow{\:\:η\:\:} φ_* \bigl(R¹ π_* Ω¹_{\wtilde{X}}\bigr) %
    → R² φ_* \bigl(π_* Ω¹_{\wtilde{X}}\bigr) → ⋯
  \]
  Recalling from \cite[Cor.~3.2]{GKK08} or more generally from
  \cite[Thm.~1.4]{GKKP11} that $π_* Ω¹_{\wtilde{X}}$ is reflexive, hence equal
  to $Ω^{[1]}_X$, Claim~\ref{claim:min0} therefore implies that $η$ is
  isomorphic.  But this is absurd, for reasons that we have already seen in the
  proof of Theorem~\ref{Fano-split}.  If $ℋ$ is an ample line bundle on $X$ and
  $c_1(π^*ℋ) ∈ H¹ \bigl(\wtilde{X},\, Ω¹_{\wtilde{X}}\bigr)$ is the Chern-class
  of its pull-back to $\wtilde{X}$, then the associated section of
  $R¹ (φ ◦ π)_*Ω¹_{\wtilde{X}}$ will clearly \emph{not} vanish\footnote{restrict
    to the fibre over a point $y ∈ Y$ where $φ$ is not isomorphic}.  The induced
  sections of $R¹ π_* Ω¹_{\wtilde{X}}$ and
  $φ_* \bigl(R¹ π_* Ω¹_{\wtilde{X}}\bigr)$ on the other hand do vanish by
  construction.  \qedhere~\mbox{(Claim~\ref{claim:minimal})}
\end{proof}

\subsection*{Step 3: Kodaira dimension}
\approvals{Daniel & yes \\Stefan & yes\\ Thomas & yes}

We aim to show that the minimal variety $X$ is quasi-Abelian, which implies in
particular that $κ(X)=0$.  As a first step, we show that $X$ is at least not of
general type.

\begin{claim}\label{claim:ngen}
  The variety $X$ is not of general type.
\end{claim}
\begin{proof}[Proof of Claim~\ref{claim:ngen}]
  Suppose to the contrary that $X$ \emph{is} of general type.  Using that $X$ is
  also minimal, recall from \cite[Thm.~1.1]{GKPT19b} that the Chern classes of
  $X$ satisfy a $ℚ$-Miyaoka-Yau inequality,
  \begin{equation}\label{eq:ace1}
    \frac{n}{2(n+1)}·[K_X]^n = %
    \frac{n}{2(n+1)}·\what{c}_1\left( Ω^{[1]}_X \right)² · [K_X]^{n-2} %
    ≤ \what{c}_2\left( Ω^{[1]}_X \right) · [K_X]^{n-2}.
  \end{equation}
  On the other hand, we assume that Equation~\eqref{eq:qcce} holds.  An
  elementary computation, using \cite[(1.14) on p.~34]{Kob87} and
  \cite[Sect.~3.8]{GKPT19b} shows that Equation~\eqref{eq:qcce} is equivalent to
  the assertion that the semistable sheaf $\sEnd(Ω^{[1]}_X)$ has vanishing
  $ℚ$-Chern classes with respect to $H$, in the sense of
  \cite[Defn.~6.1]{GKPT19}.  But then we have seen in in \cite[Fact and
  Defn.~6.5]{GKPT19} that $\sEnd(Ω^{[1]}_X)$ has vanishing $ℚ$-Chern classes
  with respect any ample bundle, and hence also with respect to any nef bundle.
  We obtain that
  \begin{equation}\label{eq:ace2}
    \frac{n-1}{2n}·[K_X]^n %
    = \frac{n-1}{2n}·\what{c}_1 \left( Ω^{[1]}_X \right)² · [K_X]^{n-2} %
    = \what{c}_2 \left( Ω^{[1]}_X \right) · [K_X]^{n-2}.
  \end{equation}
  But we know that $K_X$ is big and nef, so that $[K_X]^n > 0$.  Putting
  (in)equalities \eqref{eq:ace1} and \eqref{eq:ace2} together hence produces a
  contradiction.  \qedhere~\mbox{(Claim~\ref{claim:ngen})}
\end{proof}

\subsection*{Step 4: The Shafarevich map}
\approvals{Daniel & yes \\Stefan & yes\\ Thomas & yes}

Recall Notation~\ref{not:5-5} and consider the representation
$τ: π_1(X) → \PGL(n,ℂ)$ that defines the projectively flat structure on $ℙ$.
Write $G_X = \overline{\img(τ)} ⊆ \PGL(n,ℂ)$ for the algebraic Zariski closure
of its image.  This is a linear algebraic group, which has finitely many
components.  Passing to an appropriate étale cover of $X$, if necessary, we may
assume the following.

\begin{asswlog}\label{asswlog:conn}
  The group $G_X$ is connected.
\end{asswlog}

\begin{claim}[Representation has infinite image]\label{claim:5-x}
  The group $G_X$ is of positive dimension.
\end{claim}
\begin{proof}
  If not, then Assumption~\ref{asswlog:conn} implies that $τ$ is trivial, and
  therefore $Ω^{[1]}_X ≅ ℒ^{⊕ n}$ for a suitable Weil-divisorial sheaf $ℒ$.
  Since $K_X = \det Ω^{[1]}_X$ is nef by Claim~\ref{claim:minimal}, it follows
  that $ℒ$ is nef and hence pseudo-effective.  We may thus apply
  Theorem~\ref{split} to conclude that $ℒ$ and hence $K_X$ is numerically
  trivial, in contradiction to Assumption~\ref{asswlog:knnt}.
\end{proof}

We follow the strategy of \cite{MR3030068} and use Claim~\ref{claim:5-x} to
consider the Shafarevich map for the solvable radical of $G_X$.  We refer the
reader to \cite[Sect.~3]{MR1341589} for more on Shafarevich maps.

\begin{construction}
  Recall from Assumption~\ref{asswlog:conn} that $G_X$ is a connected, algebraic
  subgroup of $\PGL(n,ℂ)$.  Let $\Rad(G_X) \lhd G_X$ be the solvable radical of
  $G_X$; this is a normal subgroup of $G_X$.  Let $K \lhd π_1(X)$ be the kernel
  of the composed group homomorphism
  \[
    \begin{tikzcd}
      π_1(X) \ar[r, "τ"] & G_X \ar[r] & \factor{G_X}{\Rad(G_X)}.
    \end{tikzcd}
  \]
  As $K$ is normal in $π_1(X)$, we may consider an associated $K$-Shafarevich
  map, $\sh^K(X): X \dasharrow \Sh^K(X)$, where $\sh^K(X)$ is dominant and
  $\Sh^K(X)$ is a smooth, projective variety.
\end{construction}

The following two properties of $\sh^K(X)$ will be most relevant in the sequel.

\begin{fact}[\protect{Shafarevich maps are almost holomorphic, \cite[Thm.~3.6]{MR1341589}}]
  The rational map $\sh^K(X)$ is almost holomorphic.  In other words, there
  exist a Zariski open set $X° ⊆ \Def(\sh^K(X)) ⊆ X$ such that $\sh^K(X)|_{X°}$
  is well-defined and proper.  The fibres of $\sh^K(X)|_{X°}$ are connected.
  \qed
\end{fact}

\begin{fact}[\protect{Shafarevich maps and fundamental groups, \cite[Thm.~3.6]{MR1341589}}]\label{fact:Sh2a}
  Let $x ∈ X°$ be a very general point and let $Z ⊂ X$ be a subvariety through
  $x$, with normalisation $η : \wtilde{Z} → Z$.  Then, the rational map
  $\sh^K(X)$ maps $Z$ to a point if and only if the composed morphism
  \[
    \begin{tikzcd}
      π_1(\wtilde{Z}) \ar[r, "η_*"] & π_1(X) \ar[r, "τ"] & G_X \ar[r] &
      \factor{G_X}{\Rad(G_X)}
    \end{tikzcd}
  \]
  has finite image.  \qed
\end{fact}

\begin{claim}[The base of the Shafarevich map]\label{claim:6-30}
  The (smooth) variety $\Sh^K(X)$ is of general type.
\end{claim}
\begin{proof}[Proof of Claim~\ref{claim:6-30}]
  Verbatim as in \cite[Proof of Prop.~4.1]{MR3030068}: In case where
  $\factor{G_X}{\Rad(G_X)}$ is almost simple the claim follows from
  \cite[Thm.~1]{MR1384908}.  The general case easily follows by induction on the
  number of almost simple almost direct factors, and by Kawamata's results on
  the $C_{n,m}$ conjecture, \cite[Cor.~1.2]{Kawamata85}.
  \qedhere~\mbox{(Claim~\ref{claim:6-30})}
\end{proof}

Claims~\ref{claim:ngen} and \ref{claim:6-30} together imply in particular that
the fibres of $\sh^K(X)$ are positive-dimensional.  These will be investigated
next.

\subsection*{Step 5: Fibres of the Shafarevich map}
\approvals{Daniel & yes \\Stefan & yes\\ Thomas & yes}

Throughout the present step, choose a general fibre $F ⊂ X$ of $\sh^K(X)$.
There are two cases to consider: either $\Sh^K(X)$ is a point and $F = X$ is
potentially singular, or $F$ is a proper subvariety of $X$, and then it avoids
the (finitely many!) singularities of $X$,
cf.~Consequence~\ref{consequence:isolated}, and must hence be smooth.  Either
way, we will show in this step that $F$ is quasi-Abelian.  The proof is somewhat
long and therefore divided into sub-steps.

\begin{notation}
  Given an étale cover $γ :\what{F} → F$, consider the composed group morphism
  \begin{equation}\label{eq:sgsdfg}
    \begin{tikzcd}
      π_1\bigl(\what{F}\,\bigr) \ar[r] &π_1(F) \ar[r] & π_1(X) \ar[r, "τ"] &
      \PGL(n,ℂ)
    \end{tikzcd}
  \end{equation}
  and let $G_γ ⊆ \PGL(n,ℂ)$ be the Zariski closure of its image.
\end{notation}

\begin{rem}\label{rem:connected=solvable}
  Given an étale cover $γ :\what{F} → F$, the group $G_γ$ is an algebraic
  subgroup of $\PGL(n,ℂ)$ with finitely many connected components.
  Fact~\vref{fact:Sh2a} implies that the maximal connected subgroup of $G_γ$ is
  solvable.
\end{rem}

\subsection*{Step 5-1: Irregularity}
\approvals{Daniel & yes \\Stefan & yes\\ Thomas & yes}

To begin our analysis of the fibres of the Shafarevich map, we show that the
fibres have positive irregularity, at least once we pass to a suitable étale
cover.  This will later allow to analyse the fibres by means of their Albanese
map.

\begin{claim}\label{claim:uze}
  There exists an étale cover $γ : \what{F} → F$ such that
  $h⁰ \bigl(\what{F},\, Ω^{[1]}_{\what{F}}\bigr) > 0$.
\end{claim}
\begin{proof}[Proof of Claim~\ref{claim:uze}]
  Passing to a first étale cover of $F$, we may assume without loss of
  generality that $G_γ$ is connected.  If the composed map \eqref{eq:sgsdfg} has
  infinite image, then by Remark~\ref{rem:connected=solvable} this defines a
  solvable representation of $π_1\bigl(\what{F}\,\bigr)$ with infinite image.
  It follows that $H¹ \bigl(\what{F},\, ℤ \bigr)$ has positive rank and
  consequently, that $\dim_ℂ H¹\bigl(\what{F},\, ℂ \bigr)>0$.  The description
  of the natural Hodge structure on this space, \cite[Thm.~1]{Schwald16}, then
  implies that $h⁰ \bigl(\what{F},\, Ω^{[1]}_{\what{F}}\bigr) > 0$, as desired.

  It therefore remains to consider the case where the composed map
  \eqref{eq:sgsdfg} has finite image.  By Claim~\ref{claim:5-x}, this implies
  that $F$ is a proper subvariety of $X$, entirely contained in $X_{\reg}$.  We
  may then choose $γ$ such that the composed map \eqref{eq:sgsdfg} is trivial,
  implying as before that the pull-back of the projectively flat sheaf
  $Ω^{[1]}_X$ is a direct sum of the form, $γ^* Ω¹_X ≅ ℒ^{⊕ n}$, for some
  $ℒ ∈ \Pic(\what{F})$.  The pull-back of the conormal bundle sequence, which is
  non-trivial for reasons of dimension, will then read as follows,
  \[
    \begin{tikzcd}
      0 \ar[r] & \underbrace{γ^* N^*_{F/X}}_{\text{trivial bundle}} \ar[r] &
      \underbrace{γ^* Ω¹_X}_{≅ ℒ^{⊕ n}} \ar[r] & Ω¹_{\what{F}} \ar[r] & 0
    \end{tikzcd}
  \]
  Hence $H⁰\bigl(\what{F},\, ℒ\bigr) \ne \{0\}$, and consequently
  $H⁰\bigl(\what{F},\, Ω¹_{\what{F}}\big) \ne \{0\}$.
  \qedhere~\mbox{(Claim~\ref{claim:uze})}
\end{proof}

\begin{notation}\label{not:50-20}
  For the remainder of this proof, fix one particular Galois cover
  $γ : \what{F} → F$ such that $G_γ$ is connected, and where
  $h⁰ \bigl(\what{F},\, Ω^{[1]}_{\what{F}}\bigr) > 0$.  Recalling
  Remark~\ref{rem:connected=solvable}, we may apply
  Proposition~\ref{prop:pfbisr} and using its notation we write
  \begin{equation}\label{eq:OmegaFA0-1a}
    γ^* Ω^{[1]}_X ≅ ℱ ⊗ 𝒜.
  \end{equation}
  where $𝒜$ is Weil divisorial with numerical class $[𝒜] \equiv [γ^* K_X]$, and
  where $ℱ$ is locally free and admits a filtration
  \begin{equation}\label{eq:5-20-2}
    0 = ℱ_0 ⊊ ℱ_1 ⊊ ⋯ ⊊ ℱ_n = ℱ
  \end{equation}
  such that all sheaves $ℱ_i$ are locally free and flat with $\rank (ℱ_i)= i$
  and such that the sheaf $ℱ_1$ is trivial, $ℱ_1 ≅ 𝒪_X$.  Recall from
  Remark~\ref{rem:pfbisr} that there exists an inclusion
  \begin{equation}\label{eq:gjfgjh}
    𝒜 = 𝒜 ⊗ ℱ_1 ↪ γ^* Ω^{[1]}_X.
  \end{equation}
  The pull-back of the normal bundle sequence (which is trivial if $F=X$) reads
  \begin{equation}\label{eq:iopz}
    \begin{tikzcd}
      0 \ar[r] & \underbrace{γ^* N^*_{F/X}}_{\text{trivial bundle}} \ar[r] &
      \underbrace{γ^* Ω^{[1]}_X}_{= ℱ ⊗ 𝒜} \ar[r] & Ω^{[1]}_{\what{F}} \ar[r] &
      0.
    \end{tikzcd}
  \end{equation}
\end{notation}

\subsection*{Step 5-2: Numerical dimension and Kodaira dimension}
\approvals{Daniel & yes \\Stefan & yes\\ Thomas & yes}

Claim~\ref{claim:uze} has fairly direct consequences for the Kodaira dimension
of the fibres, which will turn out to be either zero or one, the second case
disappearing eventually.

\begin{claim}[Bounding the Kodaira dimension from below]\label{claim:Kodairadim2}
  Maintaining Notation~\ref{not:50-20}, if $π : \wtilde{F} → \what{F}$ is any
  resolution of singularities, then the Kodaira dimension of $\wtilde{F}$ is
  non-negative.  In particular, it follows that $κ(\what{F}) ≥ 0$.
\end{claim}
\begin{proof}[Proof of Claim~\ref{claim:Kodairadim2}]
  The assumption that $h⁰ \bigl(\what{F},\, Ω^{[1]}_{\what{F}}\bigr) > 0$,
  together with the Isomorphism~\eqref{eq:OmegaFA0-1a} and the
  Filtration~\eqref{eq:5-20-2} allows to find an index $i$ such that
  \begin{equation}\label{eq:ugh}
    h⁰\Bigl(\what{F},\, 𝒜 ⊗ \factor{ℱ_i}{ℱ_{i-1}} \Bigr) \ne 0.
  \end{equation}
  Writing $ℒ_i := \factor{ℱ_i}{ℱ_{i-1}}$ and recalling that $ℒ_i$ is invertible
  and flat, hence numerically trivial by Lemma~\ref{lem:c1offlat},
  Equation~\eqref{eq:ugh} implies in particular that the sheaf
  \[
    𝒜^{[n]} ⊗ ℒ^{⊗ n}_i = \underbrace{𝒜^{[n]} ⊗ \det ℱ\vphantom{ℒ^{⊗ n}_i}}_{=
      γ^*ω_X = ω_{\what{F}}} ⊗ \underbrace{(\det ℱ)^* ⊗ ℒ^{⊗ n}_i}_{=: ℒ'
      \text{, numerically trivial}}
  \]
  also has a non-trivial section.  In this setting, the extension theorem for
  differential forms, \cite[Thm.~4.3]{GKKP11} yields a non-trivial morphism
  $π^{[*]} ω_{\what{F}} ↪ ω_{\wtilde{F}}$ and hence a non-trivial section of
  $ω_{\wtilde{F}} ⊗ π^*ℒ'$, where again $π^*ℒ' ∈ \Pic(\wtilde{F})$ is
  numerically trivial.  The claim thus follows from the ``numerical character of
  the effectivity of adjoint line bundles'', \cite[Thm.~0.1]{CKP12}, applied to
  the smooth space $\wtilde{F}$.
  \qedhere~\mbox{(Claim~\ref{claim:Kodairadim2})}
\end{proof}

\begin{claim}[Bounding the numerical dimension from above]\label{claim:numericaldimensionbound-x}
  The numerical dimension of $\what{F}$ satisfies $ν(\what{F}) ≤ 1$.
\end{claim}
\begin{proof}[Proof of Claim~\ref{claim:numericaldimensionbound-x}]
  To begin, recall from \eqref{eq:OmegaFA0-1a} that the (nef!) numerical classes
  $n·[𝒜]$ and $[γ^* K_X] = [K_{\what{F}}]$ agree.  It follows that
  $ν(\what{F}) = ν(𝒜)$, and so it suffices to consider $ν(𝒜)$ and to show that
  $ν(𝒜) ≤ 1$.  In other words, given any very ample divisor $H$ on $\what{F}$,
  we need to show that
  \begin{equation}\label{eq:xxl}
    [𝒜]²·[H]^{n-2} = 0.
  \end{equation}

  Recall from \eqref{eq:gjfgjh} that there exists an embedding
  $𝒜 ↪ γ^* Ω^{[1]}_X$.  Combined with Sequence~\eqref{eq:iopz}, we obtain a
  sheaf morphism
  \begin{equation}\label{eq:jhk}
    𝒜 → Ω^{[1]}_{\what{F}}
  \end{equation}
  If this morphism is constantly zero, then $𝒜$ maps into the trivial sheaf
  $γ^* N^*_{F/X}$, the nef sheaf $𝒜$ must thus be trivial, and the claim is
  shown.  We will therefore continue this proof under the assumption that
  \eqref{eq:jhk} is an embedding.

  Next, choose a general tuple $(H_1, …, H_{n-2}) ∈ |H| ⨯ ⋯ ⨯ |H|$ and consider
  the associated complete intersection surface $S := H_1 ∩ ⋯ ∩ H_{n-2}$.  Then,
  $S$ has klt singularities\footnote{The surface $S$ will in fact almost always
    be smooth, except perhaps when $X = F$ and $\dim X = 2$.} and $𝒜|_S$ is
  reflexive.  The kernel of the restriction morphism for reflexive differentials
  is
  \[
    N_{S/\what{F}}^* = \ker \Bigl( Ω^{[1]}_{\what{F}}\bigr|_S → Ω^{[1]}_S \Bigr)
    ≅ 𝒪_{\what{F}}(-H)^{⊕(n-2)},
  \]
  so in particular anti-ample.  But since the nef sheaf $𝒜$ never maps to an
  anti-ample, we find that the composed sheaf morphism
  \[
    𝒜|_S → Ω^{[1]}_{\what{F}} \Bigr|_S → Ω_S^{[1]}
  \]
  cannot vanish.  The Bogomolov-Sommese vanishing theorem for the potentially
  singular surface $S$, \cite[Thm.~8.3]{GKK08}, implies that $𝒜|_S$ is not big.
  Since it is nef, we conclude that
  \[
    0 = n²·[𝒜|_S]² = n²·[𝒜]²·[H]^{n-2},
  \]
  as desired.  \qedhere~\mbox{(Claim~\ref{claim:numericaldimensionbound-x})}
\end{proof}

\begin{claim}[Kodaira dimension and numerical dimension of $F$]\label{claim:kF}
  The Kodaira-dimension and the numerical dimension of $F$ satisfy the
  inequalities
  \begin{equation}\label{eq:u56}
    0 ≤ κ(F) ≤ ν(F) ≤ 1.
  \end{equation}
\end{claim}
\begin{proof}[Proof of Claim~\ref{claim:kF}]
  Given that $γ : \what{F} → F$ is étale,
  Claim~\ref{claim:numericaldimensionbound-x} immediately implies that
  $ν(F) ≤ 1$.  Since $K_F = K_X|_F$ is nef, \cite[Prop.~2.2]{MR782236} gives the
  two rightmost inequalities in \eqref{eq:u56}, and so it remains to show that
  $κ(F) ≥ 0$.  But we already know from Claim~\ref{claim:Kodairadim2} that
  $κ(\what{F}) ≥ 0$, so that there exists a number $p ∈ ℕ$ and a non-trivial
  section $σ ∈ H⁰\bigl(\what{F},\, ω^{⊗ p}_{\what{F}}\bigr)$.  But then
  \[
    \bigotimes_{g ∈ \Gal(γ)} g^*σ ∈ H⁰\Bigl(\what{F},\, ω^{⊗ p·\#\Gal(γ)}_{\what{F}}\Bigr)
  \]
  is a non-trivial Galois-invariant pluri-form on $\what{F}$, which hence
  descends to a non-trivial pluri-form on $F$.
  \qedhere~\mbox{(Claim~\ref{claim:kF})}
\end{proof}

\subsection*{Step 5-3: Fibres of the Albanese map}
\approvals{Daniel & yes \\Stefan & yes\\ Thomas & yes}

We pointed out above that the Albanese map of $F$ is not trivial, at least once
we pass to a suitable cover.  Using the fact that the Albanese yields a
Shafarevich map for the commutator subgroup of the fundamental group, there is
much that we can say about its fibres, which will eventually turn out to be
abundant.

\begin{construction}\label{cons:itAlb-s}
  Maintaining Notation~\ref{not:50-20}, consider the Albanese map
  $\alb(\what{F}) : \what{F} → \Alb(\what{F})$ and recall from
  Claim~\ref{claim:uze} and Notation~\ref{not:50-20} that this map is
  non-trivial, which is to say that $\dim \Alb(\what{F}) > 0$.  Let
  $F_1 ⊊ \what{F}$ be a general fibre component of $\alb(\what{F})$.  This might
  be a point if $\what{F}$ is of maximal Albanese dimension.  Since $\what{F}$
  has at worst isolated singularities, $F_1$ is necessarily smooth.  Now, we use
  the fact that Albanese maps are Shafarevich maps by recalling from
  \cite[§0.1.3]{MR1341589} that the push-forward morphism of Abelianised
  fundamental groups,
  \[
    (ι_1)^{(ab)}_* : π_1( F_1 )^{(ab)} → π_1\bigl( \what{F}\, \bigr)^{(ab)}
  \]
  has finite image.  We can thus find an étale cover $γ_1 : \what{F}_1 → F_1$
  such that the composed push-forward
  \[
    (ι_1◦γ_1)^{(ab)}_* : π_1( \what{F}_1 )^{(ab)} → π_1\bigl( \what{F}\,
    \bigr)^{(ab)}
  \]
  is trivial.  The following diagram summarises the morphisms that we have
  discussed so far,
  \[
    \begin{tikzcd}[column sep=2cm, row sep=1cm]
      \what{F}_1 \ar[rrrr, bend left=10, "η"] \ar[r, two heads, "γ_1\text{, étale}"'] & F_1 \ar[r, hook, "ι_1\text{, fibre inclusion}"'] & \what{F} \ar[r, two heads, "γ\text{, étale}"'] \ar[d, "\alb(\what{F})"] & F \ar[r, hook, "ι\text{, fibre inclusion}"'] & X \ar[d, dashed, "\sh^K(X)"] \\
      && \Alb(\what{F}) && \Sh^K(X)
    \end{tikzcd}
  \]
  Observing that $\img(η) ⊊ X_{\reg}$, we consider the exact sequence
  of differentials,
  \begin{equation}\label{eq:trztu-s}
    \begin{tikzcd}
      0 \ar[r] & \underbrace{Ω¹_{F_1/X}}_{\mathclap{\text{trivial, non-zero}}}
      \ar[r] & η^* Ω¹_X \ar[r, "dη"] & Ω¹_{F_1} \ar[r] & 0.
    \end{tikzcd}
  \end{equation}
\end{construction}

\begin{claim}[Numerical invariants of $F_1$ and $\what{F}_1$]\label{claim:6-23a-s}
  Assume the setting of Construction~\ref{cons:itAlb-s}.  Then,
  $0 ≤ κ(F_1) ≤ ν(F_1) ≤ 1$ and $0 ≤ κ(\what{F}_1) ≤ ν(\what{F}_1) ≤ 1$.
\end{claim}
\begin{proof}[Proof of Claim~\ref{claim:6-23a-s}]
  Given that the morphism $γ_1$ is étale, it suffices to consider the variety
  $F_1$ only.  We will also assume that $F_1$ is not a point, or else there is
  little to show, cf.~Claim~\ref{claim:kF}.  On the one hand, using that
  $\img(η) ⊊ X_{\reg}$, we obtain a description of the canonical bundle,
  \[
    ω_{F_1} \overset{\text{\eqref{eq:trztu-s}}}{≅} η^*(ω_X)
    \overset{\text{\eqref{eq:OmegaFA0-1a}}}{≅} η^* \bigl(𝒜^{[⊗ n]} ⊗ \det
    ℱ\bigr).
  \]
  Observing that $η^* 𝒜$ is invertible and that $η^* ℱ$ is flat, this gives a
  numerical equivalence $[ω_{F_1}] \equiv n·[η^* 𝒜]$.  On the other hand,
  \eqref{eq:gjfgjh} and \eqref{eq:trztu-s} combine to give a morphism
  \[
    α : η^* 𝒜 → Ω¹_{F_1}.
  \]
  There are two cases to consider.  If the morphism $α$ is not zero,
  Claim~\ref{claim:6-23a-s} follows from \cite[Lem.~3.1]{MR3030068}.  Otherwise,
  if the morphism $α$ is zero, we obtain an embedding of the nef line bundle
  $η^* 𝒜$ into the trivial sheaf $Ω¹_{F_1/X}$.  It follows that $η^* 𝒜$ and
  hence $ω_{F_i}$ are numerically trivial, so $ν(F_1) = 0$.  By
  \cite[Thm.~8.2]{Kawamata85}, this implies $κ(F_1) = 0$.  Either way,
  Claim~\ref{claim:6-23a-s} follows.
  \qedhere~\mbox{(Claim~\ref{claim:6-23a-s})}
\end{proof}

\begin{claim}[Semiampleness of the varieties $\what{F}_1$ and $F_1$]\label{claim:6-25a-s}
  Assume the setting of Construction~\ref{cons:itAlb-s}.  Then, the canonical
  bundles $K_{\what{F}_1}$ and $K_{F_1}$ are semiample.
\end{claim}
\begin{proof}[Proof of Claim~\ref{claim:6-25a-s}]
  As before, we consider the variety $\what{F}_1$ only and assume that it is not
  a point.  We also consider the following commutative diagram of group
  morphisms,
  \[
    \begin{tikzcd}[column sep=2cm, row sep=1cm]
      π_1\bigl( \what{F}_1 \bigr) \ar[r, "τ◦ η_*"] \ar[d, "\txt{\scriptsize $(ι_1◦γ_1)_*$\\\scriptsize trivial after Abelianisation}"'] & G_X \ar[r, hook, "\text{subset}"] & \PGL(n,ℂ) \\
      π_1\bigl( \what{F} \bigr) \ar[r, "τ◦ (ι◦γ)_*"'] & G_γ \ar[u, hook, "\text{subset}"']
    \end{tikzcd}
  \]
  The representation $τ◦η_*$ is solvable because it factors via the solvable
  group $G_γ$.  Better still, the representation $τ◦η_*$ is in fact trivial, or
  else there would be a non-trivial Abelian subrepresentation, contradicting
  triviality of the Abelianised push-forward $(ι_1◦γ_1)^{(ab)}_*$.  As a
  consequence, we find that $η^* Ω¹_X$ is isomorphic to a direct sum of line
  bundles,
  \[
    ∃\, ℒ ∈ \Pic(\what{F}_1) : η^* Ω¹_X ≅ ℒ^{⊕ n}.
  \]
  Recalling from Sequence~\eqref{eq:trztu-s} that $η^* Ω¹_X$ contains the
  trivial, non-zero subbundle $Ω¹_{\what{F}_1/X}$, we find that $ℒ$ admits a
  section.  But then $η^* Ω¹_X$ is generically generated, and so is its quotient
  $Ω¹_{\what{F}_1}$.  By \cite[Rem.~2.2]{Fuj13}, the manifold $\what{F}_1$ is
  then of maximal Albanese dimension.  But then \cite[Thm.~4.2]{Fuj13} implies
  that the nef canonical bundle $K_{\what{F}_1}$ is semiample, as claimed.
  \qedhere~\mbox{(Claim~\ref{claim:6-25a-s})}
\end{proof}

\subsection*{Step 5-4: Abundance and description of $F$}
\approvals{Daniel & yes \\Stefan & yes\\ Thomas & yes}

In this step, we finally prove abundance for $F$ and describe its Iitaka
fibration.  According to Claim~\ref{claim:kF}, there are only two cases,
$κ(F) = 0$ and $κ(F) = 1$, which we consider separately.

\begin{claim}\label{claim:solvK0}
  If $κ(F)=0$, then $F$ is quasi-Abelian.
\end{claim}
\begin{proof}[Proof of Claim~\ref{claim:solvK0}]
  To keep this proof readable, we consider three cases separately.
  \begin{enumerate}
  \item\label{il:A} The divisor $K_F$ is numerically trivial and $F = X$.
    
  \item\label{il:B} The divisor $K_F$ is numerically trivial and $F ⊊ X$.
    
  \item\label{il:C} The divisor $K_F$ is not numerically trivial.
  \end{enumerate}
  Case~\ref{il:A} is easiest.  If $K_X$ is numerically trivial, then
  Equation~\eqref{eq:qcce} guarantees that
  \[
    \what{c}_2 \bigl( Ω^{[1]}_X \bigr) · [H]^{n-2} = 0.
  \]
  Then, $X$ is quasi-Abelian by \cite[Thm.~1.2]{LT18} and we are done.  In
  Case~\ref{il:B}, apply Item~\ref{il:6-5-2} of Lemma~\ref{lem:6-5} to the
  inclusion $F ⊂ X_{\reg}$ to find again that $F$ is quasi-Abelian.

  We consider Case~\ref{il:C} for the remainder of the proof, with the aim of
  producing a contradiction.  To begin, we consider a strong log resolution of
  $\what{F}$ and the Stein factorisation of the associated Albanese morphism.
  We obtain a commutative diagram as follows
  \[
    \begin{tikzcd}[column sep=2cm]
        & \wtilde{F} \ar[r, "a\text{, conn.~fibers}"] \ar[d, "π\text{, strong log res.}"'] & S \ar[r, "b\text{, finite}"] & \Alb(\wtilde{F}) \ar[d, "≅"] \\
      F & \what{F} \ar[rr, "alb_{\what{F}}"'] \ar[l, "γ\text{, étale}"] && \Alb(\what{F}).
    \end{tikzcd}
  \]
  By Claim~\ref{claim:uze}, $\Alb(\what{F})$ is not a point.  A general fibre
  $\wtilde{F}_1 ⊊ \wtilde{F}$ of $a$ is therefore smooth.
  
  Observe that the image $F_1 = π(\wtilde{F_1})$ is a connected component of a
  general fibre of $\alb_{\what{F}}$, which avoids the (finitely many)
  singularities of $\what{F}$.  The restriction
  $π|_{\wtilde{F}_1} : \wtilde{F}_1 → F_1$ is thus isomorphic, and
  Claim~\ref{claim:6-25a-s} applies to show that $\wtilde{F}_1$ is a good
  minimal model.  More is true.  Since $κ (F) = 0$ by assumption, we deduce from
  Claim~\ref{claim:Kodairadim2} that $κ(\wtilde{F}) = 0$.  Hence it follows from
  \cite[Thm.~4.2]{Lai11} that $\wtilde{F}$ has a good minimal model
  $\what{F}_{\min}$, which has terminal singularities and numerically trivial
  canonical class.

  We use the existence of $\what{F}_{\min}$ to describe $K_{\what{F}}$ in more
  detail.  To this end, resolve the singularities of the rational map
  $\what{F}_{\min} \dasharrow \what{F}$ and obtain morphisms as follows,
  \[
    \begin{tikzcd}[column sep=1.5cm]
      \what{F}_{\min} \ar[d, "\alb_{\what{F}_{\min}}"'] & Y \ar[r, "π_1", "\text{birational}"'] \ar[l, "π_2"', "\text{birational}"] \ar[d, "\alb_{Y}"] & \what{F} \ar[d, "\alb_{\what{F}}"] \\
      \Alb(\what{F}_{\min}) \ar[r, equal] & \Alb(Y) \ar[r, equal] & \Alb(\what{F})
    \end{tikzcd}
  \]
  The fact that $\what{F}_{\min}$ has terminal singularities implies that
  $K_Y \equiv \what{E}$, where $\what E ∈ ℚ\Div(Y)$ is effective with
  $π_2$-exceptional support.  Recalling that every fibre of $π_2$ is rationally
  connected by \cite[Cor.~1.5]{HMcK07}, we find that every component of
  $\supp \what{E}$ is uniruled.  The canonical class $K_{\what{F}}$ is then
  numerically equivalent to $E := (π_1)_*\what{E}$.  The $ℚ$-divisor $E$ is
  effective, nef and not trivial (because $K_{\what{F}}$ is not numerically
  trivial), and its components are again uniruled.  But the Albanese map
  $\alb_{\what{F}}$ contracts all rational curves in $\supp E$!  We claim that
  the image set $\alb_{\what{F}}(\supp E)$ cannot contain any isolated points.
  \begin{itemize}
  \item If $\dim \alb_{\what{F}}(\what{F}) ≥ 2$, this follows from nefness of
    $E \equiv K_{\what{F}} = γ^* K_X$ and from the Negativity
    Lemma~\vref{lem:negativity1}.
    
  \item If $\dim \alb_{\what{F}}(\what{F}) = 1$, this follows from nefness of
    $E \equiv K_{\what{F}} = γ^* K_X$, from the assumption that
    $0 = κ(F) = κ(\what{F})$ and from the Negativity
    Lemma~\vref{lem:negativity2}.
  \end{itemize}
  Either way, we find that $E$ contains rational curves that stay away from the
  finitely many singular points of $\what{F}$.  Since $K_X$ is nef, this
  contradicts Item~\ref{il:6-5-1} of Lemma~\ref{lem:6-5}.
  \qedhere~\mbox{(Claim~\ref{claim:solvK0})}
\end{proof}

\begin{claim}[Description of $F$ if $κ(F)=1$]\label{claim:solvK1}
  If $κ(F)=1$, then $K_F$ is semiample and the general fibre of the Iitaka
  fibration is quasi-Abelian.
\end{claim}
\begin{proof}[Proof of Claim~\ref{claim:solvK1}]
  If $κ(F)=1$, then $ν(F) = κ(F)$ by Claim~\ref{claim:kF}.  It follows that
  $K_F$ is semiample, \cite[Cor.~6-1-13]{KMM87}.  As $F$ has isolated
  singularities at worst, the general fibre $A$ of the associated Iitaka
  fibration $F → B$ is smooth, contained in $X_{\reg}$, and has numerical
  trivial canonical class.  As before, use Item~\ref{il:6-5-2} of
  Lemma~\ref{lem:6-5} to find that $A$ is quasi-Abelian.
  \qedhere~\mbox{(Claim~\ref{claim:solvK1})}
\end{proof}

\subsection*{Step 6: Abundance for $X$, description of the Iitaka fibration}
\approvals{Daniel & yes \\Stefan & yes\\ Thomas & yes}

We have seen in the previous step that the general fibres of the Shafarevich map
are abundant.  As we will see now, this implies Abundance for $X$.

\begin{claim}[Abundance for $X$]\label{claim:AbdX}
  The canonical bundle $K_X$ is semiample.
\end{claim}
\begin{proof}[Proof of Claim~\ref{claim:AbdX}]
  If $\Sh^K(X)$ is a point, then $F = X$ and we have seen in
  Claim~\ref{claim:kF} that $0 ≤ κ(F) = κ(X) ≤ 1$.  We know that abundance holds
  because the two cases have been described in Claims~\ref{claim:solvK0} and
  \ref{claim:solvK1}, respectively.  If $\Sh^K(X)$ is positive-dimensional, then
  we can apply Proposition~\ref{prop:abundance} to the Shafarevich map.
  \begin{itemize}
  \item Assumption~\ref{il:AS1} is satisfied by Claim~\ref{claim:6-30}
    
  \item Assumption~\ref{il:AS2} is satisfied as nefness of $K_X$ has been shown
    in Claim~\ref{claim:minimal}.
    
  \item Assumption~\ref{il:AS3} holds since $X$ has only isolated singularities,
    see Consequence~\ref{consequence:isolated}.

  \item Assumption~\ref{il:AS4} has been shown in Claims~\ref{claim:solvK0} and
    \ref{claim:solvK1}, respectively.
  \end{itemize}
  The claim thus follows.  \qedhere~\mbox{(Claim~\ref{claim:AbdX})}
\end{proof}

\begin{rem}
  Claim~\ref{claim:AbdX} implies in particular that
  \begin{itemize}
  \item the Iitaka fibration has a positive-dimensional base by
    Assumption~\ref{asswlog:knnt}, and
    
  \item the fibers of the Iitaka fibration are quasi-Abelian by
    Item~\ref{il:6-5-2} of Lemma~\ref{lem:6-5}.
  \end{itemize}
\end{rem}

We denote the Iitaka fibration by $\iit(X) : X → \Iit(X)$.

\begin{rem}[Ambro’s canonical bundle formula]\label{rem:ACBF}
  Ambro’s canonical bundle formula for projective klt pairs,
  \cite[Thm.~4.1]{MR2134273} or \cite[Thm.~3.1]{MR2944479}, provides us with an
  effective $ℚ$-divisor $Δ$ on $\Iit(X)$ that makes the pair
  $\bigl( \Iit(X), Δ \bigr)$ klt.
\end{rem}

\subsection*{Step 7: Abelian group schemes}
\approvals{Daniel & yes \\Stefan & yes\\ Thomas & yes}

Following the ideas of Jahnke-Radloff, \cite{MR3030068}, we use Kollár's
characterisation of étale quotients of Abelian group schemes to show that the
Iitaka fibration of a suitable étale cover of $X$ is birational to an Abelian
group scheme.

\begin{claim}\label{claim:glfg}
  There exists an étale cover $γ : \what{X} → X$ whose Iitaka fibration is
  birational to an Abelian group scheme over a smooth projective base admitting
  a level three structure.  More precisely, there exists a commutative diagram
  \begin{equation}\label{eq:glfg}
    \begin{tikzcd}[column sep=2cm]
      A \ar[r, dashed, "φ\text{, birational}"] \ar[d, "α"] & \what{X} \ar[r, "\text{étale}"] \ar[d, "\iit(\what{X})"] & X \\
      S \ar[r, "ψ\text{, birational}"'] \ar[u, bend left, "σ"] &
      \Iit(\what{X})
    \end{tikzcd}
  \end{equation}
  where $S$ is smooth and projective, and $α$ is a family of Abelian varieties
  admitting a level three structure.
\end{claim}
\begin{proof}[Proof of Claim~\ref{claim:glfg}]
  We construct from right to left a diagram as follows,
  \begin{equation}\label{eq:K12}
    \begin{tikzcd}[column sep=2cm]
      A \ar[d, "α"'] \ar[r, "\text{étale}"] & \wcheck{A} \ar[d, "\wcheck{α}"'] \ar[r, dashed, "\text{birational}"] & \wcheck{X} \ar[d, "\wcheck{a}"'] \ar[r, "γ\text{, étale}"] & X \ar[d, "\iit(X)"] \\
      S \ar[r, "γ_S \text{, étale}"'] & \wcheck{S} \ar[r, "ρ\text{, birational}"'] & \wcheck{Y} \ar[r, "β\text{, finite}"'] & \Iit(X)
    \end{tikzcd}
  \end{equation}
  where $\wcheck{a}$ has connected fibers, where $α$ and $\wcheck{α}$ are
  Abelian group schemes and where $α$ has the additional structure of a family
  of Abelian varieties with level three structure.  To begin, we claim that the
  variety $X$ has generically large fundamental group along the general fibre
  $X_y$ of the Iitaka fibration.  In case where $X$ is smooth, this claim has
  been shown in \cite[Prop.~5.1]{MR3030068}.  We leave it to the reader to check
  that in our case, where $X$ has only finitely many singularities, the proof of
  \cite[Prop.~5.1]{MR3030068} still applies verbatim.  Given that the fibers of
  the Iitaka fibration are quasi-Abelian, Proposition~\ref{prop:K63} will
  therefore apply to yield the right and middle square of Diagram~\eqref{eq:K12}
  --- strictly speaking, Proposition~\ref{prop:K63} gives $ρ$ only as a rational
  map, but we can always blow up to make it a morphism, and then pull back
  $\wcheck{A}$.  Once this is done, we can find a suitable étale cover
  $S → \wcheck{S}$ such that the pull-back $A := \wcheck{A} ⨯_{\wcheck{S}} S$ is
  an Abelian group scheme with level three structure.  We note the following
  property for later reference, which holds by construction.
  \begin{enumerate}
  \item\label{il:birat} If $s ∈ S$ is any general point, with image
    $y ∈ \Iit(X)$, then the induced map between fibres, $A_s \dasharrow X_y$
    is birational.
  \end{enumerate}
  
  We still need to construct $\what{X}$.  We begin its construction with the
  observation that $K^{[m]}_X = \iit(X)^* (H)$ for a suitable number $m ∈ ℕ^+$
  and a suitable ample $H ∈ \Div\bigl(\Iit(X)\bigr)$.  But then,
  \begin{equation}\label{eq:hjfg}
    K^{[m]}_{\wcheck{X}} = γ^* K^{[m]}_X = γ^* \iit(X)^* (H) = \wcheck{a}^{\,*} (β^* H),
  \end{equation}
  where $β^* H$ is ample on $\what{Y}$; this shows that $\wcheck{a}$ is the
  Iitaka fibration for $\wcheck{X}$ and allows to apply Ambro's canonical bundle
  formula, Remark~\ref{rem:ACBF}, which shows that $\wcheck{Y}$ is the
  underlying space of a klt pair.  We are interested in this, because
  \cite[Thm.~1.1]{Takayama2003} then asserts that
  $π_1\bigl(\wcheck{S}\,\bigr) = π_1\bigl( \wcheck{Y}\,\bigr)$.  As a
  consequence, we find that the Stein factorisation of the morphism
  $S → \wcheck{Y}$ factors via an étale cover of $\wcheck{Y}$ and therefore
  gives a diagram,
  \[
    \begin{tikzcd}[column sep=2cm]
      A \ar[r, "φ"', dashed] \ar[rr, dashed, bend left=15] \ar[d, "α"'] & \what{X} \ar[r, "\text{étale}"'] \ar[d] & \wcheck{X} \ar[d, "\wcheck{a}"] \\
      S \ar[r, "ψ \text{, birational}"] \ar[rr, "ρ ◦ γ_S"', bend right=15] & \what{Y} \ar[r, "γ_Y\text{, étale}"] & \wcheck{Y},
    \end{tikzcd}
  \]
  where $\what{X} := \wcheck{X} ⨯_{\wcheck{Y}} \what{Y}$ and where the map $φ$
  is induced by the universal property of the fibre product.  Observe, exactly
  as in \eqref{eq:hjfg}, that the natural projection map $\what{X} → \what{Y}$
  is the Iitaka fibration of $\what{X}$,
  $\iit(\what{X}) : \what{X} → \Iit(\what{X}) = \what{Y}$.
  Assertion~\ref{il:birat} implies that $φ$ is birational.
  \qedhere~\mbox{(Claim~\ref{claim:glfg})}
\end{proof}

\begin{claim}\label{claim:glfg2}
  There exists an étale cover $γ : \what{X} → X$ whose Iitaka fibration is
  birational to an Abelian group scheme over $\Iit(\what{X})$ with level three
  structure.  In other words, there exists a commutative diagram
  \begin{equation}\label{eq:glfg2}
    \begin{tikzcd}[column sep=2cm]
      A \ar[r, dashed, "φ\text{, birational}"] \ar[d, "α"] & \what{X} \ar[d, "\iit(\what{X})"] \\
      \Iit(\what{X}) \ar[r, equal] \ar[u, bend left, "σ"] & \Iit(\what{X})
    \end{tikzcd}
  \end{equation}
  where $α$ is a family of Abelian varieties with level three structure.
\end{claim}
\begin{proof}[Proof of Claim~\ref{claim:glfg2}]
  Consider Diagram~\eqref{eq:glfg} and recall that Abelian varieties with level
  three structure admit a fine moduli space $𝒜_3$ with universal family
  $\sU_3 →𝒜_3$, and that moreover $𝒜_3$ does not contain any rational curve
  \cite[Lem.~5.9.3]{Kollar93}.  On the other hand, while observing
  Remark~\ref{rem:ACBF} recall from \cite[Cor.~1.5]{HMcK07} that the fibres of
  the morphism $ψ$ are rationally chain connected.  It follows that the moduli
  map $S → 𝒜_3$ factors via $ψ$ to give a morphism $\Iit(\what{X}) → 𝒜_3$.  To
  conclude, set $A := \sU_3 ⨯_{𝒜_3} \Iit(\what{X})$.
  \qedhere~\mbox{(Claim~\ref{claim:glfg2})}
\end{proof}

\subsection*{Step 8: End of proof}
\approvals{Daniel & yes \\Stefan & yes\\ Thomas & yes}

We end the proof by showing that $\what{X}$ itself is an Abelian group scheme
over a smooth base with ample canonical bundle.

\begin{claim}\label{claim:6-38}
  In the setting of Claim~\ref{claim:glfg2}, the rational map
  $τ := φ^{-1} : \what{X} \dasharrow A$ is a morphism.
\end{claim}
\begin{proof}[Proof of Claim~\ref{claim:6-38}]
  Given that $\what{X}$ is klt and that the fibres of $α$ do not contain any
  rational curves, the claim follows from \cite[Cor.~1.6]{HMcK07}.
  \qedhere~\mbox{(Claim~\ref{claim:6-38})}
\end{proof}

\begin{claim}\label{claim:6-39}
  The base variety $\Iit(\what{X})$ is smooth.  In particular, $A$ is smooth.
\end{claim}
\begin{proof}[Proof of Claim~\ref{claim:6-39}]
  Assume not.  Let $s ∈ \Iit(\what{X})_{\sing}$ be any singular point, and let
  $a ∈ α^{-1}(s)$ be a general point of the fibre.  Note that the morphism $α$
  is smooth, so $a$ will be a singular point of $A$.  On the other hand, note
  that the general choice of $a$ guarantees that $τ^{-1}(a)$ does not contain
  any of the (finitely many) singular points of $X$.  It follows that $τ$
  resolves the singularity at $a$, and that the fibre $τ^{-1}(a)$ is thus
  necessarily positive-dimensional.

  There is more that we can say.  Recalling from Remark~\ref{rem:ACBF} that
  $\bigl( \Iit(X), Δ \bigr)$ is klt for a certain $ℚ$-divisor $Δ$ it follows
  from smoothness of $α$ that $(A, α^*Δ)$ is klt as well.  But then
  $τ^{-1}(a) ⊊ \iit(\what{X})^{-1}(s) ∩ \what{X}_{\reg}$ is covered by rational
  curves, violating Item~\ref{il:6-5-1} of Lemma~\ref{lem:6-5} from above.
  \qedhere~\mbox{(Claim~\ref{claim:6-39})}
\end{proof}

\begin{claim}\label{claim:6-40}
  The morphism $τ$ is isomorphic and $\what{X}$ is therefore smooth.
\end{claim}
\begin{proof}[Proof of Claim~\ref{claim:6-40}]
  If $a ∈ A$ is any (smooth!) point where $τ$ is not isomorphic, then
  $τ^{-1}(a)$ is positive-dimensional, and every section of $K_{\what{X}}$ will
  necessarily vanish along $τ^{-1}(a)$.  That contradicts the semiampleness of
  $K_{\what{X}}$, Claim~\ref{claim:AbdX}.
  \qedhere~\mbox{(Claim~\ref{claim:6-40})}
\end{proof}

In the current, now nonsingular setup, the following has already been observed
in \cite[Proof of Lemma~5.1]{MR3030068}.

\begin{claim}\label{claim:6-41}
  The canonical bundle of $\Iit(\what{X})$ is ample.
\end{claim}
\begin{proof}[Proof of Claim~\ref{claim:6-41}]
  The smooth variety $\Iit(\what{X})$ is of general type, while the existence of
  $σ$ and minimality of $X$ and hence $\what X$ together imply that its
  canonical bundle is nef, hence abundant.  Fibres of the Iitaka fibration for
  $\Iit(\what{X})$ are covered by rational curves, which do not exist by the
  existence of $σ$ and the fact that $\what X$ (as an étale cover of $X$) does
  not contain rational curves by Item~\ref{il:6-5-1} of Lemma~\ref{lem:6-5}.
\end{proof}

Now that we know that $\what{X}$ is an Abelian group scheme over a smooth
projective base with ample canonical bundle, the final step of the argument of
Jahnke-Radloff, \cite[Thm.~6.1]{MR3030068}, applies to conclude also our proof.
\qed~\mbox{(Theorem~\ref{thm:charQAbelian})}

%
%
\svnid{$Id: 06-extra.tex 651 2020-10-13 07:33:36Z kebekus $}

\section{Proof of Theorem~\ref*{thm:6-1}}
\subversionInfo

\subsection*{Step 1: Preparations}
\approvals{Daniel & yes \\Stefan & yes\\ Thomas & yes}

In the setup of Theorem~\ref{thm:6-1}, we set
\[
  ℱ := \bigl(\Sym^nΩ¹_X ⊗ 𝒪_X(-K_X)\bigr)^{**}.
\]
Let $H$ be a very ample divisor on $X$ and let
$(D_1, …, D_{n-1}) ∈ |H|^{⨯ (n-1)}$ be a general tuple of divisors, with
associated general complete intersection curve $C := D_1 ∩ ⋯ ∩ D_{n-1}$.  The
curve $C$ is then smooth and entirely contained in $X_{\reg}$.  The restriction
$ℱ|_C$ is locally free and nef, with $c_1(ℱ|_C) = 0$.  It follows that $ℱ|_C$ is
semistable and therefore that $ℱ$ is semistable with respect to $H$.  But then,
$Ω^{[1]}_X$ will likewise be semistable with respect to $H$.

\subsection*{Step 2: Local freeness of $ℱ$}
\approvals{Daniel & yes \\Stefan & yes\\ Thomas & yes}

Next, choose a resolution of singularities $π: \wtilde{X} → X$ such that the
quotient sheaf
\[
  \wtilde{ℱ} := \factor{π^*(ℱ)}{\tor}
\]
is locally free; such a resolution exists by \cite[Thm.~3.5]{Rossi68}.  Recall
from Fact~\ref{fact:1} that since $ℱ$ is nef, then so are $π^*(ℱ)$, $\wtilde{ℱ}$
and $\det \wtilde{ℱ}$.  There is more that we can say.  Since the determinant of
$ℱ$ is trivial by construction, $\det ℱ ≅ 𝒪_X$, we may write
\[
  \det \wtilde{ℱ} = 𝒪_{\wtilde X}\left(\sum a_i·E_i\right) \quad \text{where
    $E_i$ are $π$-exceptional and }a_i ∈ ℤ.
\]
The divisor $\sum a_i·E_i$ is nef.  But then the Negativity Lemma in the form
\cite[Lem.~2.1]{LWY20} asserts that $a_i ≤ 0$ for all $i$, which of course means
that all $a_i = 0$.  It follows that $\det \wtilde{ℱ} = 𝒪_{\wtilde X}$ and thus
that $\wtilde{ℱ}$ is numerically flat.  This has two consequences.
\begin{itemize}
\item First, it follows from the descent theorem for vector bundles on
  resolutions of klt spaces, \cite[Thm.~1.2]{GKPT17}, that $ℱ$ is locally free,
  numerically flat, and that $\wtilde{ℱ} = π^*(ℱ)$.
    
\item Second, it follows from \cite[Prop.~3.7]{GKP20a} that $Ω¹_{X_{\reg}}$ is
  projectively flat.
\end{itemize}

\subsection*{Step 3: Singularities of $X$}
\approvals{Daniel & yes \\Stefan & yes\\ Thomas & yes}

If $\what{X} → X$ is any maximally quasi-étale cover, then local freeness of $ℱ$
immediately implies local freeness and nefness for the reflexive normalised
cotangent sheaf of $\what{X}$.  The covering space $\what{X}$ therefore
reproduces the assumptions of Theorem~\ref{thm:6-1}, which allows to assume
without loss of generality that $X$ is itself maximally quasi-étale.  Together
with projective flatness of $Ω¹_{X_{\reg}}$, this assumption allows to apply the
local description of projectively flat sheaves found in
\cite[Prop.~3.11]{GKP20a}: every singular point $x ∈ X^{(an)}$ admit a
neighbourhood $U$, open in the analytic topology, such that
\[
  Ω^{[1]}_U \simeq ℒ_U^{⊕ n}
\]
with some Weil-divisorial sheaf $ℒ_U$ on $U$.  It follows from
Proposition~\ref{prop:icq} that $X$ has at worst isolated singularities.

\subsection*{Step 4: End of proof in case where $X$ is higher-dimensional}
\approvals{Daniel & yes \\Stefan & yes\\ Thomas & yes}

If $n ≥ 3$, consider the general complete intersection surface
$S := D_1 ∩ ⋯ ∩ D_{n-2}$.  Using that $X$ has isolated singularities, we find
that $S$ is smooth and contained in $X_{\reg}$, so that $Ω¹_X|_S$ is locally
free, $H$-semistable and projectively flat.  But by \cite[Prop.~3.1.b on
p.~42]{Kob87} this implies
\[
  \frac{n-1}{2n}·c_1\bigl(Ω¹_X|_S\bigl)² = c_2\bigl(Ω¹_X|_S\bigl)
\]
and therefore
\[
  \frac{n-1}{2n}·\what{c}_1(X)² · [H]^{n-2} = \what{c}_2(X) · [H]^{n-2}.
\]
The claim thus follows from the semistability of $Ω¹_X$ and
Theorem~\ref{thm:charQAbelian}.

\subsection*{Step 5: End of proof in case where $X$ is a surface}
\approvals{Daniel & yes \\Stefan & yes\\ Thomas & yes}

It remains to consider the case where $\dim X = 2$, so that $X$ is a surface
with klt quotient singularities.  We have seen above that
\[
  \Sym^{[2]} Ω¹_X = ℱ ⊗ 𝒪_X(K_X),
\]
where $ℱ$ is a locally free numerically flat sheaf of rank three.  The second
$ℚ$-Chern class is therefore computed as
\begin{align*}
  \what c_2 \left( \Sym^{[2]} Ω^{[1]}_X \right) & = \what{c}_2 \bigl(ℱ ⊗ 𝒪_X(K_X) \bigr) = \what{c}_2 \bigl(𝒪_X(K_X)^{⊕ 3} \bigr) = 3·[K_X]².  \\
  \intertext{On the other hand, the standard formula for the second Chern class of a symmetric product gives}
  \what c_2 \left( \Sym^{[2]} Ω^{[1]}_X \right) & = 2·[K_X]² + 4·\what{c}_2(X).
\end{align*}
Comparing these two equations, we find $\frac{1}{4}·[K_X]² = \what c_2(X)$.
Thanks to the semistability of $Ω^{[1]}_X$, we may again apply
Theorem~\ref{thm:charQAbelian} and end the proof.  \qed

\medskip

\end{document}